\documentclass[a4paper, 10pt]{article}
\usepackage{amssymb} 
\usepackage{mathtools} 
\usepackage{amsthm} 
\usepackage[pdftex]{graphicx}
\usepackage[pdftex]{color}
\usepackage{natbib}
\usepackage{subcaption}
\usepackage[shortlabels]{enumitem}
\usepackage{geometry}
\newcount\fromtop
\newfont{\smcal}{cmu10 scaled 1200}
\newfont{\handw}{cmmi10 scaled 1200}
\newfont{\handws}{cmmi10 scaled 800}
\newtheorem{Prop}{Proposition}[section]
\newtheorem{Lem}[Prop]{Lemma}
\newtheorem{As}[Prop]{Assumption}

\newtheorem{Th}[Prop]{Theorem}
\newtheorem{Rm}[Prop]{Remark}
\newtheorem{Def}[Prop]{Definition}

\newtheorem{Cor}[Prop]{Corollary}

\newtheorem{Ex}[Prop]{Example}
\newcommand{\Span}{\mbox{\rm span}}
\newcommand{\grad}{\mbox{\rm grad}}
\newcommand{\cov}{\mbox{\rm Cov}}

\newcommand{\rank}{\mbox{\rm rank}}

\newcommand{\fA}{\mathfrak{A}}

\newcommand{\cG}{{\cal G}}

\newcommand{\cN}{{\mathcal{N}}}
\newcommand{\cR}{{\cal R}}
\newcommand{\E}{\mathbb E}
\newcommand{\EE}{\mathbb E}
\newcommand{\RR}{\mathbb R}
\newcommand{\NN}{\mathbb N}
\newcommand{\DD}{\mathbb D}
\newcommand{\SSS}{\mathbb S}
\newcommand{\Prb}{\mathbb P}
\newcommand{\as}{\mbox{\rm ~a.s.}}

\newcommand{\tr}{\mbox{\rm trace}}

\newcommand{\Hess}{\mbox{\rm Hess\,}} 

\DeclareMathOperator*{\argmin}{\mbox{\rm argmin}}

\newcommand{\lw}{\mbox{\handw \symbol{96}}}

\newcommand{\iid}{\operatorname{\stackrel{i.i.d.}{\sim}}}
\newcommand{\twovec}[2]{\begin{pmatrix} #1 \\ #2 \end{pmatrix}}
\newcommand{\twoclass}[2]{\begin{bmatrix} #1 \\ #2 \end{bmatrix}}

\newcommand{\wv}{\widetilde{v}}
\newcommand{\ww}{\widetilde{w}}

\newcommand\commentout[1]{}

\begin{document}
\title{Backward Nested Descriptors Asymptotics with\\ Inference on Stem Cell Differentiation}
\author{Stephan F. Huckemann\footnote{Felix-Bernstein-Institut f\"ur Mathematische Statistik in den Biowissenschaften, Georg-August-Universit\"at G\"ottingen} 
~and Benjamin Eltzner$^*$} 
\date{}
\maketitle

\begin{abstract}
  For sequences of random backward nested subspaces as occur, say, in dimension reduction for manifold or stratified space valued data, asymptotic results are derived. In fact, we formulate our results more generally for backward nested families of descriptors (BNFD). Under rather general conditions, asymptotic strong consistency holds. Under additional, still rather general hypotheses, among them existence of a.s. local twice differentiable charts, asymptotic joint normality of a BNFD can be shown. If charts factor suitably, this leads to individual asymptotic normality for the last element, a principal nested mean or a principal nested geodesic, say. It turns out that these results pertain to principal nested spheres (PNS) and principal nested great subsphere (PNGS) analysis by \cite{Jung2010} as well as to the intrinsic mean on a first geodesic principal component (IMo1GPC) for manifolds and Kendall's shape spaces. A nested bootstrap two-sample test is derived and illustrated with simulations. In a study on real data, PNGS is applied to track early human mesenchymal stem cell differentiation over a coarse time grid and, among others, to locate a change point with direct consequences for the design of further studies. 
\end{abstract}

\noindent%

\noindent%
\section{Introduction}

In this paper, the novel statistical problem of deriving asymptotic results for nested random sequences of statistical descriptors for data in a non-Euclidean space is considered. 
It can be viewed as a generalization of classical PCA's asymptotics, e.g. by \cite{A63,Wa83,RuYa97}, where, as a consequence of Pythagoras' theorem, nestedness of approximating subspaces is trivially given and thus requires no special attention. For PCA analogs for data in non-Euclidean spaces, due to curvature, nestedness considerably complicates design of descriptors and, to the best knowledge of the authors, has hindered any asymptotic theory to date.

For dimension reduction of non-Euclidean data, \emph{Procrustes analysis} by \cite{Gow} and later \emph{principal geodesic analysis} by \cite{fletch6} are approaches to mimic PCA on shape spaces and Riemannian manifolds, respectively. Both build on the concept of a Fr\'echet mean, a minimizer of expected squared distance, around which classical PCA is conducted for the data mapped to a suitable tangent space. Asymptotics for such means have been subsequently provided, among others, by \cite{Z77, HL96,BP03,BP05,H_Procrustes_10}, allowing for inferential methods such as two-sample tests. Asymptotics for these \emph{tangent space PCA} methods, however, reflecting the \emph{forward nestedness} due to random basepoints (i.e. corresponding means) of tangent spaces with random PCs therein, remain open to date.

Moreover, these tangent space PCA methods are in no way canonical. Not only may statistical outcomes depend on specific choices of tangent space coordinates, more severely, given curvature, no tangent space coordinates can correctly reflect mutual data distances. For this reason, among others, \emph{geodesic principal component analysis} (GPCA) has been introduced by \cite{HZ06,HHM07}, \emph{iterated frame bundle development} by \cite{Sommer2013} and \emph{barycentric subspaces} by \cite{Pennec2015,Pennec2016}. As the following example teaches, nestedness may be lost.
\begin{Ex}\label{intro.ex}
 Consider data on a two-sphere that is confined to its equator and nearly uniformly spread out on it. Then the best $L^2$ approximating geodesic is the equator and far away there are two (due to symmetry) intrinsic Fr\'echet means, each close to one of the poles, see \cite{H_meansmeans_12}. 
\end{Ex}
Let us now detail our ideas, first by elucidating the following.

\textbf{Classical PCA from a geometric perspective.}
Given data on $Q=\RR^m$, for every $0\leq k\leq m$ a unique affine subspace $p^k$ of dimension $k$ is determined by equivalently minimizing residual sums of squares or, among those containing the classical mean $\mu$, maximizing the projected variance. Also equivalently, these subspaces have representations as $p^k = \mu + \Span\{\gamma_1,\ldots,\gamma_k\}$, the affine translates 
of spans from an eigenvector decomposition $\gamma_1,\ldots,\gamma_m$ of the data's covariance matrix with descending eigenvalues. In consequence, one may either start from the zero dimensional mean and subsequently add most descriptive dimensions (forward) or start from the full dimensional space and remove least descriptive dimensions (backward) to obtain the same forward and backward nested sequence of subspaces
\begin{align}\label{backward-nested-subspaces:eq}
  \{\mu\}=p^0 \subset p^1\subset \ldots \subset p^m = Q\,.
\end{align} 
For non-Euclidean data, due to failure of Pythagoras' theorem, this canonical decomposition of data variance is no longer possible. For a detailed discussion see \cite{HHM07,Jung2010}. 

\textbf{Nestedness of non-Euclidean PCA} is highly desirable, when due to curvature and data spread, intrinsic Fr\'echet means are away from the data. For instance in Example \ref{intro.ex}, in order to have a mean on the equator, also in this case, \cite{JFM2011} devised \emph{principal arc analysis} with the \emph{backward nested mean} confined to the best approximating circle. This method and its generalization \emph{backward nested sphere analysis} (PNS) by \cite{Jung2010} give a tool for descriptive shape analysis that often strikingly outperforms tangent space PCA, e.g. \cite{PJGZCDHM13}. Here, the data space is a unit sphere $Q=\mathbb S^m$ of dimension $m\in \NN$, say, and in \eqref{backward-nested-subspaces:eq} each of the $p^k$ is a $k$-dimensional (small) subsphere for PNS and for \emph{principal nested great spheres} (PNGS) it is a $k$-dimensional great subsphere. In passing we note that PNS is \emph{higher dimensional} in the sense of having higher dimensional descriptor spaces than classical PCA and PNGS which are equally high dimensional, cf. \cite{HE_LASR15}.

To date, however, there is no asymptotic theory for PNS available, in particular there are no inferential tools for backward nested means, say. Asymptotic results for non-data space valued descriptors, geodesics, say, are only available for single descriptors (cf. \cite{H_ziez_geod_10, H14SemiIntrinsic}) that are directly defined as minimizers, not indirectly as a nested sequence of minimizers.

\textbf{Challenges for and results of this paper.}
It is the objective of this paper to close this gap by providing asymptotic results for rather general random \emph{backward nested families of descriptors} (BNFDs) on rather general spaces. The challenge here is that random objects that are constrained by other random objects are to be investigated, requiring an elaborate 
setup.  Into this setup, we translate strong consistency arguments of \cite{Z77} and \cite{BP03}, and introducing a \emph{constrained} M-estimation technique, we show joint asymptotic normality of an entire BNFD. In the special case of nested subspaces, BNFDs may terminate at any dimension and $p_0 = \{\mu\}$ is not required.

As we minimize a functional under the constraining conditions that other functionals are minimized as well, our approach can be called \emph{constrained M-estimation}. In the literature, this term \emph{constrained M-estimation} has been independently introduced by \cite{KentTyler1996} who robustify M-estimators by introducing constraining conditions and by \cite{Geyer1994, Shapiro2000}, who consider M-estimators that are confined to closed subsets of a Euclidean space with specifically regular boundaries. It seems that our M-estimation problem, which is constrained to satisfying other M-estimation problems has not been dealt with before. We solve it using a random Lagrange multiplier approach.

Furthermore, in order to obtain asymptotic normality of each single sequence element, in particular for the last, we require the rather technical concept of \emph{factoring charts}. Our very general setup will be illustrated, still with some effort, by example of PNS, PNGS and the \emph{intrinsic mean on a first geodesic principal component} (IMo1GPC).

In order to exploit nested asymptotic normality for a \emph{nested two-sample test}, we utilize bootstrapping techniques. While for Fr\'echet means, as they are descriptors assuming values in the data space, one can explicitly model the dependence of the random base points of the tangent spaces as in \cite{HHM09}, so that suitable statistics can be accordingly directly approximated, this modeling and approximation can be avoided using the bootstrap as in \cite{BP05}. For our application at hand, as data space and descriptor space are different, we cannot approximate the distribution of random descriptors and we fall back on the bootstrap.

\textbf{Suggestions for live imaging of stem cell differentiation.}
After illustrations of our nested two-sample test by simulations for PNS and PNGS, we apply it to a cutting edge application in adult human stem cell differentiation research. ``Rooted in a line of experimentation originating in the 1960s'' (from \cite{Biancoetal2013}), the promise that stem cells taken from a patient's bone marrow may be used to rebuild specific, previously lost, patient's tissue is currently undergoing an abundance of clinical trials. Although the underlying mechanisms are, to date, not fully understood, it is common knowledge that early stem cell differentiation is triggered by biomechanical cues, e.g. \cite{ZemelRehfeldBrownDischerSafran2010NatPhys}, which result in specific ordering of the cellular \emph{actin-myosin filament skeleton}. In collaboration with the Third Institute of Physics at the University of G\"ottingen we map fluorescence images of cell structures to two-spheres, where each point stands for a specific ordering. With our 2D PNGS two-sample test we can track the direction of increased ordering over the first 24 hours. We find, however, a consistent reversal of ordering between hours 16 to 20 which hint toward the effect of cell division. This effect suggests that the commonly used time point of 24 hours for fixated hMSCs imaging, e.g. as in \cite{ZemelRehfeldBrownDischerSafran2010NatPhys}, may not be ideal for cell differentiation detection. In fact, our method can be used to direct more elaborate and refined imaging techniques, such as time resolved in-vivo cell imaging, using \cite{EltznerWollnikGottschlichHuckemannRehfeldt2015}, say, to investigate specifically discriminatory time intervals in detail.

\textbf{We conclude our introduction} by noting that our setup of BNFDs has a canonical form on a Riemannian manifold with $p^k$ in (\ref{backward-nested-subspaces:eq}) being a totally geodesic submanifold, not necessarily of codimension one in $p^{k+1}$, however. For example for Kendall's shape spaces $\Sigma_2^j$ which is a complex projective space of real dimension $m=2(j-2)$, cf. \cite{K84}, we have a sequence of
\begin{align}\label{nested-shape-spaces:eq}
  \begin{array}{rcccccccccccccl}
    \{\mu\}&=&p^0 &\subset& p^1&\subset& \Sigma_2^{3}& \subset &\Sigma_{2}^{4}&\subset&\ldots &\subset &\Sigma_2^j &=& Q\\
    &&&1&&1&&2&&2&&2
  \end{array}
\end{align}
where the numbers below the inclusions denote the corresponding co-dimensions. 

More generally, we believe that our setup can be generalized to Riemann stratified spaces. For example, (\ref{nested-shape-spaces:eq}) generalizes at once to $\Sigma_r^j$ (the shape space of $r$-dimensional $j$ landmark configurations which has dimension $r(j-1)-1 - r(r-1)/2$) with $3\leq r\leq j-1$, cf. \cite{KBCL99}, now with
\begin{align*}
  \begin{array}{rcccccccccccccl}
    \{\mu\}&=&p^0 &\subset& p^1&\subset& \Sigma_r^{r+1}& \subset &\Sigma_{r}^{r+2}&\subset&\ldots &\subset &\Sigma_r^j &=& Q\,.\\
    &&&1&&\frac{r(r+1)}{2}-2&&r&&r&&r
  \end{array}
\end{align*}
Indeed, in Section \ref{scn:Kendalls-shape-spaces} we illustrate the generalization to the sequence $\{\mu\}=p^0 \subset p^1\subset p^2=\Sigma_r^j$, giving the 
IMo1GPC for arbitrary $\Sigma_r^j$. Our setup may also generalize to phylogenetic tree spaces as introduced by \cite{BilleraHolmesVogtmann2001}, cf. also \cite{BardenLeOwen2013}, or torus-PCA and the more general polysphere-PCA, cf. \cite{EltznerHuckemannMardia15},\cite{EltznerHuckemannJung15}. Moreover, our setup may be applied to flags of barycentric subspaces as introduced by \cite{Pennec2016}. 

\textbf{Our paper is organized as follows.} In the following section we introduce the abstract setup of BNFDs and show that the essential assumptions are fulfilled for PNS, PNGS and IMo1GPCs for Riemannian manifolds and Kendall's shape spaces. In the section to follow we will develop a set of assumptions necessary for the main results on asymptotic strong consistency and normality which are stated in Section \ref{main-results:scn}. Also in Section \ref{main-results:scn}, we give our nested bootstrap two-sample test. The elaborate proof of asymptotic strong consistency is deferred to the Appendix. In Section \ref{application:scn} we show simulations and our applications to stem cell differentiation.


\section{Backward Nested Families of Descriptors}\label{setup:scn}

In this section we first introduce the general framework including the fundamental assumption of \emph{factoring charts} which is essential to prove asymptotic normality of single nested descriptors in Section \ref{main-results:scn}. Then we give examples: the intrinsic mean on a first geodesic principal component (IMo1GPC) for Riemannian manifolds, principal nested spheres (PNS) as well as principal nested great spheres (PNGS) and finally we give an example for the IMo1GPCs also on non-manifold Kendall's shape spaces. The first example is rather straightforward, the last is slightly more involved and the second and third are much more involved. The differential geometry used here can be found in any standard textbook, e.g. \cite{Lee2013}.

First, let us quickly sketch the ideas in case of IMo1GPCs on a Riemannian manifold $Q$. There, we have the space $P_1$ of point sets of geodesics on $Q$ which is the first non-trivial descriptor space ``below'' the space $P_2=\{Q\}$. In order to show strong asymptotic consistency in Theorem \ref{SC:thm}, on $P_1\times P_1$ we require the concept of a \emph{loss function} $d_1$ that has some properties of a distance between two (point sets of) geodesics. In order to model nestedness, given a geodesic $p\in P_1$ we require the set $S_p$ of lower dimensional descriptors in $P_0 = Q$ which lie on $p$. These are the candidate nested means on $p$, and in this case, $S_p=p$. Further, we need the data projection $\pi_{Q,p}:Q\to p$ and we measure the distance $\rho\big(\pi_{Q,p}(q),s\big)$ of the projected data to a candidate nested mean $s \in S_p$. Then every  $(Q,p,s)$ with $p\in P_1,s\in S_p$ will be a \emph{backward nested family of descriptors} (BNFD) and the set $(p,s)$ with $p\in P_1,s\in S_p$ carries a natural manifold structure. It is the objective of \emph{factoring charts} to represent this manifold locally as a direct product of arbitrary variable offsets $s\in P_0=Q$ times a suitable space parametrizing directions of geodesics, parametrized independently from the offset $s\in Q$, cf. Figure \ref{chartfactoring:fig}. We will see that is precisely the geometry of the projective bundle. Once we establish asymptotic normality of the backward nested descriptor $(p,s)$, asymptotic normality follows at once also for $s$, because, under factoring charts, $s$ is given by some coordinates of a Gaussian vector as reasoned in the proof of Theorem \ref{CLT:thm}.

\subsection{General Framework}\label{general_framework:scn}

With a silently underlying probability space $(\Omega,\fA,\Prb)$, \emph{random elements} on a topological space $Q$ are mappings $X:\Omega \to Q$ that are measurable with respect to the Borel $\sigma$-algebra of $Q$. In the following, \emph{smooth} refers to existing continuous 2nd order derivatives. 

For a topological space $Q$ we say that a continuous function $d:Q\times Q\to [0,\infty)$ is a \emph{loss function} if $d(q,q') =0$ if and only if $q=q'$. We say that a set $A\subset Q$ is $d$-bounded if $\sup_{a,a'\in A}d(a,a') <\infty$. Moreover, we say that $B\subset Q$ is \emph{$d$-Heine Borel} if all closed $d$-bounded subsets of $B$ are compact.

\begin{Def}\label{BNFD:def}
  A separable topological space $Q$, called the \emph{data space}, admits \emph{backward nested families of descriptors} (BNFDs) if
  \begin{enumerate}[(i)]
    \item there is a collection $P_j$ ($j=0,\ldots,m$) of topological separable spaces with loss functions $d_j : P_j\times P_j \to [0,\infty)$;
    \item $P_m = \{Q\}$; 
    \item every $p\in P_j$ ($j=1,\ldots,m$) is itself a topological space and gives rise to a topological space $\emptyset \neq S_p \subset P_{j-1}$ which comes with a continuous map
    \begin{align*}
      \rho_p:p\times S_p \to [0,\infty)\,;
    \end{align*}
    \item for every pair $p\in P_j$ ($j=1,\ldots,m$) and $s\in S_p$ there is a measurable map called \emph{projection}
    \begin{align*}
      \pi_{p,s}: p \to s\,.
    \end{align*}
  \end{enumerate}
  For $j\in \{1,\ldots,m\}$ and $k \in \{1,\ldots,j\}$ call a family
  \begin{align*}
    f = \{p^{j},\ldots,p^{j-k}\},\mbox{ with }p^{l-1} \in S_{p^{l}}, l=j-k+1,\ldots,j
  \end{align*}
  a \emph{backward nested family of descriptors (BNFD) from $P_j$ to $P_{j-k}$}. The space of all BNFDs from $P_j$ to $P_{j-k}$ is given by
  \begin{align*}
    T_{j,k} = \left\{f=\{p^{j-l}\}_{l=0}^k: p^{l-1} \in S_{p^{l}}, l=j-k+1,\ldots,j\right\} \subseteq \prod_{l=0}^k \limits P_{j-l}\,.
  \end{align*}
  For $k\in \{1,\ldots,m\}$, given a BNFD $f = \{p^{m-l}\}_{l=0}^{k}$ set
  \begin{align*}
    \pi_{f} = \pi_{p^{m-k+1},p^{m-k}} \circ \ldots \circ \pi_{p^m,p^{m-1}} : p^m \to p^{m-k}\,
  \end{align*}
  which projects along each descriptor.
  For another BNFD $f' = \{{p'}^{j-l}\}_{l=0}^{k}\in T_{j,k}$ 
  set
  \begin{align*}
    d^j(f,f') = \sqrt{\sum_{l=0}^{k} d_j(p^{j-l},{p'}^{j-l})^2}\,.
  \end{align*}
\end{Def}

In case of PNS, the nested projection $\pi_f$ is illustrated in Figure \ref{pns_illustration:fig} (a).

\begin{Def}
  Random elements $X_1,\ldots,X_n \iid X$ on a data space $Q$ admitting BNFDs give rise to \emph{backward nested population} and \emph{sample means} (abbreviated as BN means)
  \begin{align*}
    \{E^{f^{j}}:j=m,\ldots,0\},\quad \{E^{f_n^{j}}_n:j=m,\ldots,0\}
  \end{align*}
  recursively defined via $E^m = \{Q\} = E^m_n$, i.e. $p^m = Q= p_n^m$ and
  \begin{align*}
    E^{f^{j-1}}&=\argmin_{s\in S_{p^{j}}}\E[\rho_{p^{j}}(\pi_{f^{j}}\circ X,s)^2], & f^j &= \{p^k\}_{k=j}^{m} \\
    E^{f_n^{j-1}}_n &=\argmin_{s\in S_{p^{j}_n}} \sum_{i=1}^n\rho_{p_n^{j}}(\pi_{f_n^{j}}\circ X_i,s)^2, & f_n^j &= \{p_n^k\}_{k=j}^{m} \,.
  \end{align*}
  where $p^{j} \in E^{f^{j}}$ and $p^{j}_n \in E^{f_n^{j}}$ is a measurable choice for $j=1,\ldots,m$.
  
  We say that a BNFD $f=\{p^k\}_{k=0}^{m}$ gives \emph{unique} BN population means if $E^{f^{j}} = \{p^{j}\}$ with $f^j= \{p^k\}_{k=j}^{m}$ for all $j=0,\ldots,m$. 
  
  Each of the $E^{f^{j-1}}$ and $E^{f_n^{j-1}}_n$ is also called a \emph{generalized Fr\'echet mean}.
\end{Def}

Note that by definition there is only one $p^m = Q\in P_m$. For this reason, for notational simplicity, we ignore it from now on and begin all BNFDs with $p^{m-1}$ and consider thus the corresponding $T_{m-1,k}$.

\begin{Def}[Factoring Charts]\label{def:factoring}
  Let $j\in\{0,\ldots,m-1\}, k\in \{1,\ldots,j\}$. If $T_{j,k}$ and $P^{j-k}$ carry smooth manifold structures near ${f'} = ( {p'}^{j},\ldots, {p'}^{j-k}) \in T_{j,k}$ and ${p'}^{j-k}\in P^{j-k}$, respectively, with open $W\subset T_{j,k}$, $U\subset P^{j-k}$ such that ${f'} \in W$, ${p'}^{j-k}\in U$, and with local charts 
  \begin{align*}
    \psi: W\to \RR^{\dim(W)},~f=({p}^{j},\ldots, {p}^{j-k})\mapsto \eta = (\theta,\xi),\quad \phi: U\to\RR^{\dim(U)},~p^{j-k}\mapsto \theta\,
  \end{align*}
  we say that the \emph{chart $\psi$ factors}, cf. Figure \ref{chartfactoring:fig} (a) and (b), if with the projections
  \begin{align*}
    \pi^{P^{j-k}}: T_{j,k}\to P^{j-k},~f \mapsto p^{j-k},\quad \pi^{\RR^{\dim(U)}} : \RR^{\dim(W)} \to \RR^{\dim(U)},~(\theta,\xi) \mapsto \theta
  \end{align*}
  we have
  \begin{align*}
    \phi\circ \pi^{P^{j-k}}|_W = \pi^{\RR^{\dim(U)}}|_{\psi(W)}\circ \psi\,.
  \end{align*}
\end{Def}

\begin{figure}[ht!]
  \centering
  $\vcenter{\hbox{\subcaptionbox{BNFD}[0.27\textwidth]{\includegraphics[width=0.27\textwidth, clip=true, trim=8cm 4cm 3.5cm 1.5cm]{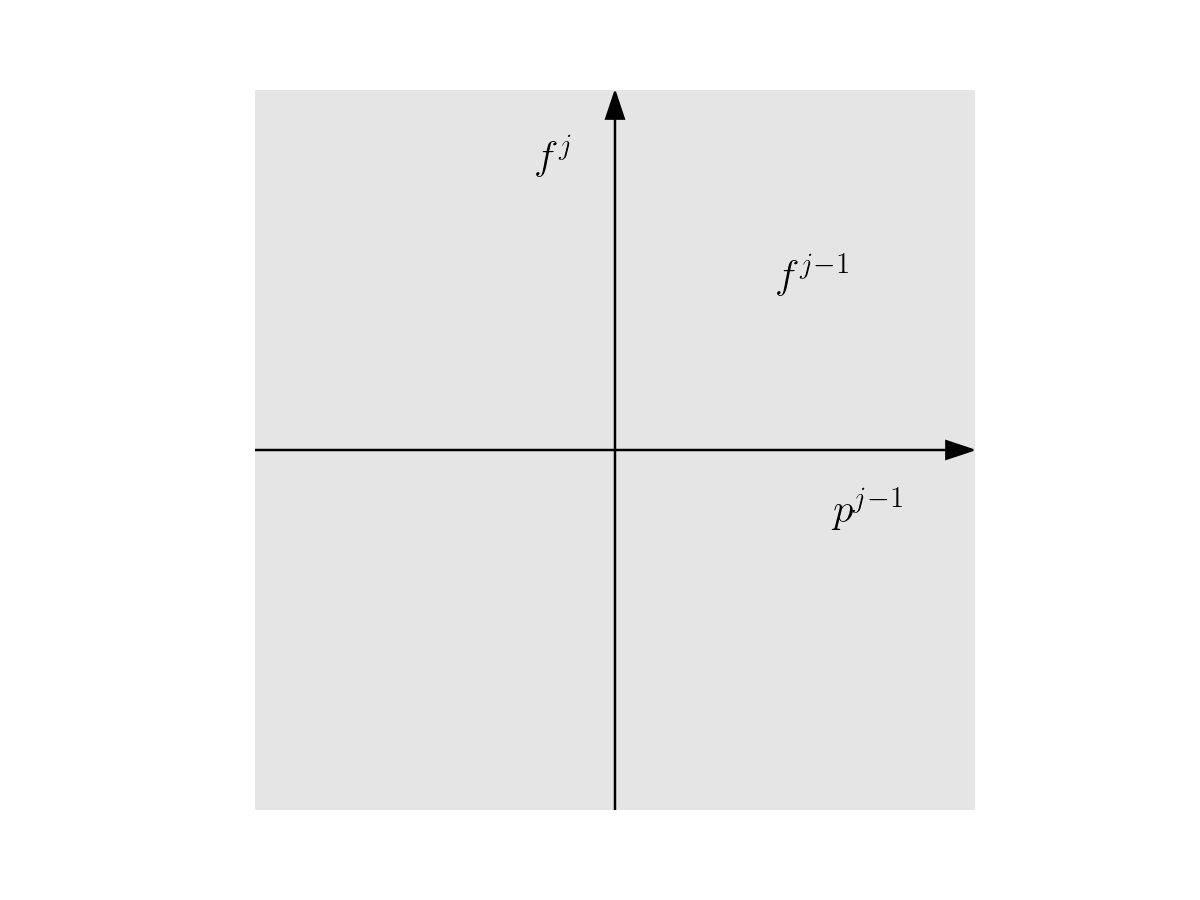}}}}
  \scalebox{2}{$\stackrel{\psi}{\to}$}~~ 
  \vcenter{\hbox{\subcaptionbox{Coordinates}[0.27\textwidth]{\includegraphics[width=0.27\textwidth, clip=true, trim=8cm 4cm 3.5cm 1.5cm]{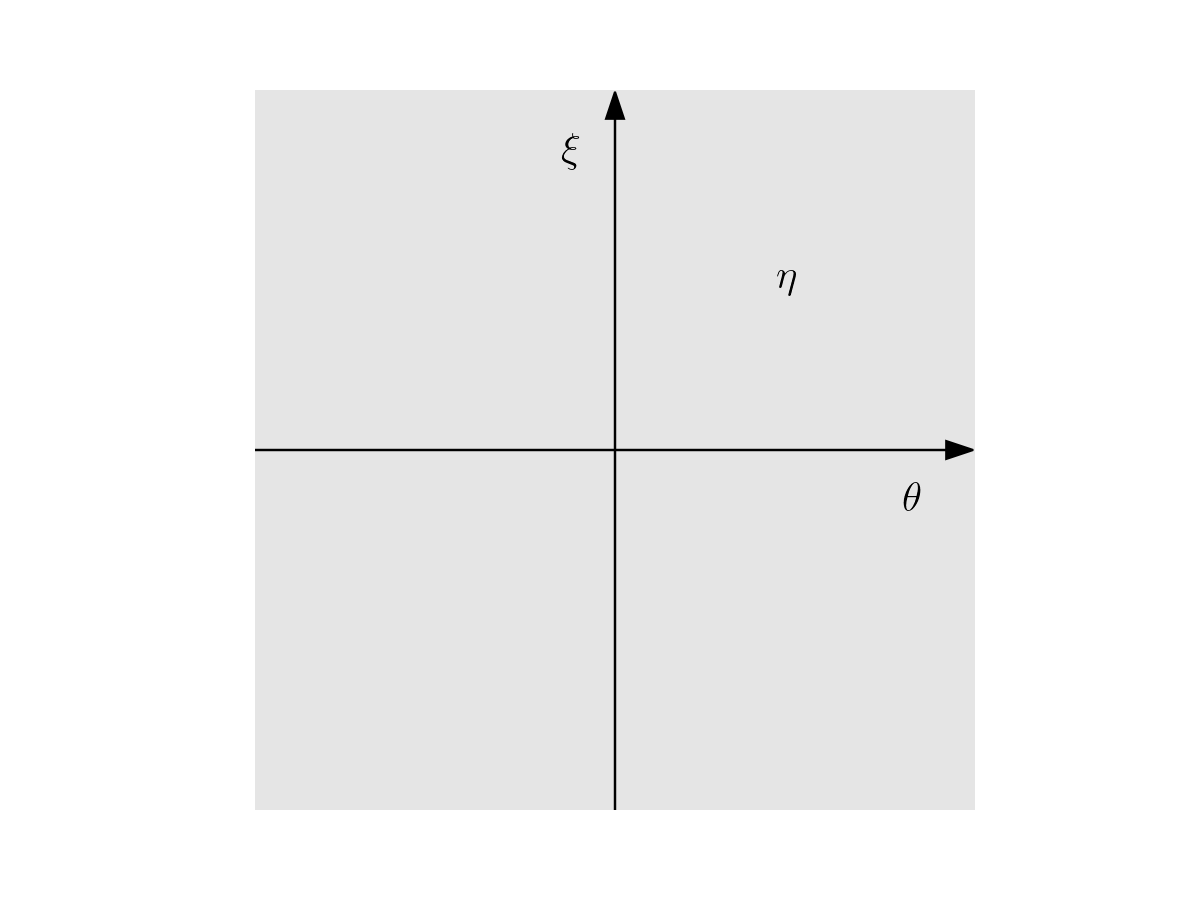}}}}
  \hspace*{0.02\textwidth}
  \vcenter{\hbox{\subcaptionbox{Projective bundle}[0.35\textwidth]{\includegraphics[width=0.35\textwidth, clip=true, trim=4cm 2.5cm 3cm 1.5cm]{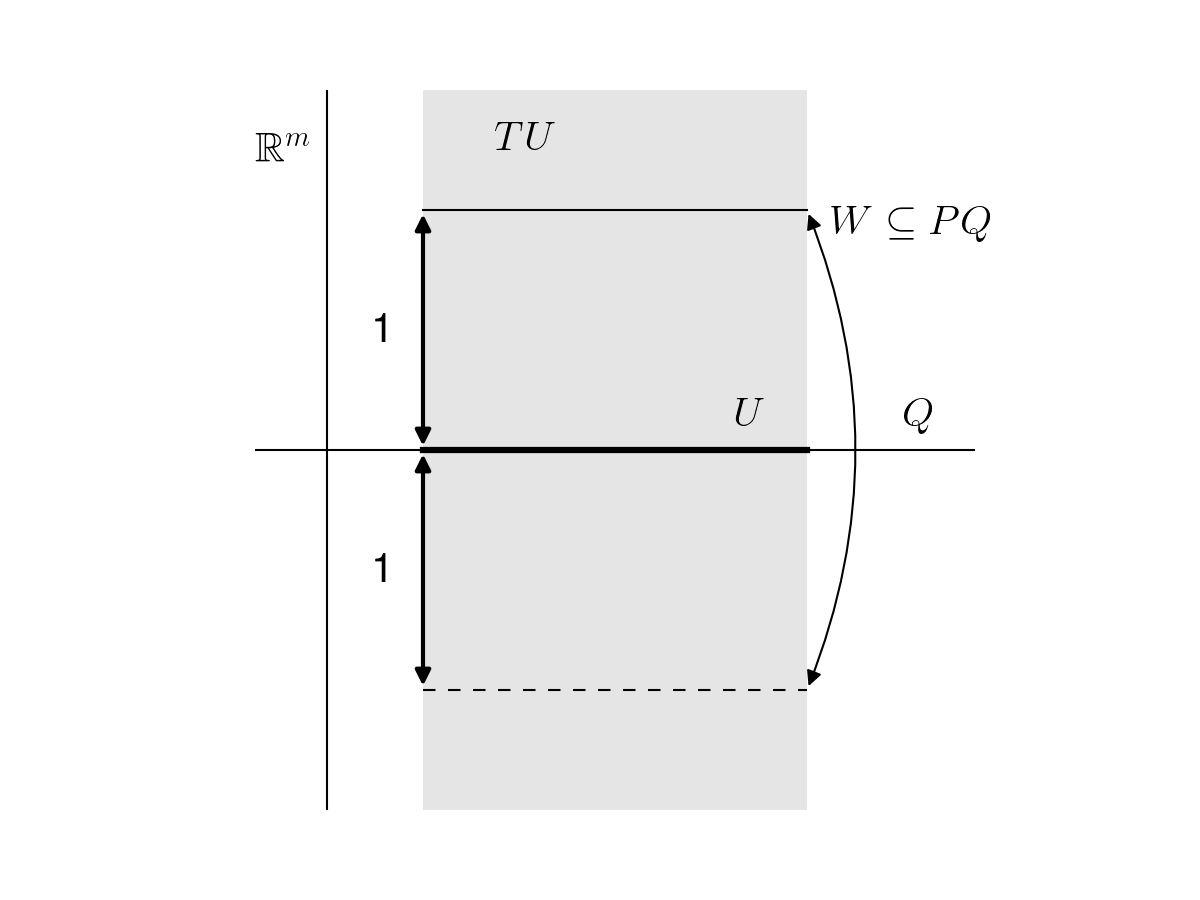}}}}$
  \caption{\it Factoring charts.\label{chartfactoring:fig}}
\end{figure}


\subsection{Intrinsic Mean on a First Principal Component Geodesic for Manifolds}\label{IMo1GPC:scn}

Suppose that $X_1,\ldots,X_n \sim X$ are random variables assuming values on a Riemannian manifold $Q$ with Riemannian norm $\|\cdot\|$ for the tangent spaces $T_qQ$ ($q\in Q$), induced metric $d:Q\times Q \to [0,\infty)$, \emph{projective tangent bundle} 
$PQ = \{(q, \{v,-v\}): q \in Q, v\in T_qQ, \|v\| = 1\}$ and space of classes of geodesics given by their point sets
\begin{align*}
  P_1 = \{[\gamma_{q,v}]: (q,\{v,-v\})\in PQ\}, ~ [\gamma_{q,v}] = \{\gamma_{r,w}: \gamma_{q,v}(t) = r,~ \dot\gamma_{q,v}(t)= w\mbox{ for some }t\}
\end{align*}
where $t\mapsto \gamma_{q,v}(t)$ denotes the unique maximal geodesic through $q=\gamma_{q,v}(0)$ with unit speed velocity $v=\dot\gamma_{q,v}(0)$, $\|v\|=1$. 
Then consider
\begin{align*}
  P_2 = \{Q\},\quad S_Q =P_1,\quad P_0 = Q\,.
\end{align*}

There is a well defined distance between a point $s\in Q$ and a class of geodesics determined by
\begin{align*}
  \rho_Q: Q\times P_1 \to [0,\infty),~\big(s,[\gamma_{q,v}]\big) \mapsto \inf_{t}d\big(s,\gamma_{q,v}(t))\,.
\end{align*}
Then every class of geodesics determined by
\begin{align*}
   \argmin_{(q,v)\in TQ}\EE\Big[\rho\big(X,[\gamma_{q,v}]\big)^2\Big]\mbox{ or }\argmin_{(q,v)\in TQ}\sum_{k=1}^n \rho\big(X_k,[\gamma_{q,v}]\big)^2
\end{align*}
is called a \emph{first population principal component geodesic} or a \emph{first sample principal component geodesic}, respectively, cf. \cite{HZ06}. 

Moreover, given a first population principal component geodesic $p=[\gamma_{q,v}]$ and a first sample principal component geodesic $p_n=[\gamma_{q_n,v_n}]$, with the orthogonal projection
\begin{align*}
  \pi_{Q,p} : Q \to p = [\gamma_{q,v}], q'\mapsto \argmin_{\gamma_{q,v}(t)} d(q',\gamma_{q,v}(t))
\end{align*}
which is well defined outside a set of zero Riemannian volume, e.g. \cite[Theorem 2.6]{HHM07}, we have the \emph{intrinsic population means on $p$} and  \emph{intrinsic sample means on $p_n$}  determined by
\begin{align*}
  \argmin_{\gamma_{q,v}(t)}\EE\Big[\rho_{p_n}\big(\pi_{Q,p_n}\circ X,\gamma_{q,v}(t)\big)^2\Big]\mbox{ or } \argmin_{\gamma_{q,v}(t)}\sum_{k=1}^n \rho_p\big(\pi_{Q,p}\circ X_k,\gamma_{q,v}(t)\big)^2\,,
\end{align*}
respectively, where $\rho_p(q,q') = d(q,q')$ for $q,q'$ in $p$. In particular, we have the space of \emph{backward nested descriptors}
\begin{align*}
  T_{1,1} = \{(p,s): p = [\gamma_{q,v}] \in P_1,~ s \in p\}
\end{align*}
which carries the natural manifold structure of the projective tangent bundle $PQ$ conveyed by the identity 
\begin{align}\label{eq:factoring-GPCA}
  T_{1,1} \to PQ,~([\gamma_{q,v}],s) \mapsto (s,\{w,-w\})
\end{align}
where $w= \dot\gamma_{q,v}(t)$, $\|w\|=1$, if $s=\gamma_{q,v}(t)$.

Recall that the tangent bundle $TQ= \{(q,v): q\in Q, v\in T_qQ\}$ admits \emph{local trivializations}, i.e. every $q\in Q$ has a local neighborhood $U\subset Q$ with a smooth one-to-one mapping
\begin{align*}
  \tau = (\tau_1,\tau_2): TU \to U \times \RR^{\dim(Q)}\,
\end{align*}
where the first coordinate satisfies $\tau_1(q',v') = q'$ for all $v'\in T_{q'}Q$, $q'\in U$ and the second coordinate $\tau_2$ is a vector space isomorphism.
In consequence, for a given $(q,v) \in PQ$, with local charts $\phi: U \to \RR^{\dim(Q)}$ of $Q$ around $q$, and $\chi : H \to  \RR^{\dim(Q)-1}$ of the real projective space $P\RR^{\dim(Q)-1}$ of dimension ${\dim(Q)}-1$ around $\{\tau_2(q,v), - \tau_2(q,v)\}\in H \subset \RR^{\dim(Q)-1}$, $H$ open, and the open set
\begin{align*}
  W =\Big \{\big(q',\{v',-v'\}\big): (q',v') \in PU, \big\{\tau_2(q',v'), - \tau_2(q',v')\big\}\in H \Big\}\subset PQ\,,
\end{align*}
the mapping
\begin{align*}
  \psi: W \to \RR^{\dim(Q)} \times \RR^{\dim(Q)-1},~~\big(q',\{v',-v'\}\big) \mapsto \Big(\phi(q'),\chi\big\{\tau_2(q',v'), - \tau_2(q',v')\big\}\Big)
\end{align*}
yields a local chart that factors as in Definition \ref{def:factoring}. This scenario is sketched in Figure \ref{chartfactoring:fig} (c).

\subsection{The Geometric Framework of PNS and PNGS}\label{sec:PNS-Framework}
The (nested) projections detailed below are illustrated in Figure \ref{pns_illustration:fig}.

\textbf{Notation.}
Consider the standard inner product $\langle \cdot,\cdot \rangle$ on $\RR^{m+1}$ with norm $\|x\| = \langle x,x\rangle^{1/2}$ and the $m$-dimensional unit sphere $Q=\SSS^m = \{x\in \RR^{m+1}:\|x\| = 1\}$, with interior $\mathbb D^{m+1} = \{x\in \mathbb R^{m+1}: \|x\| <1\}$. For any matrix $v=(v_1,\ldots,v_k ) \in \RR^{m\times k}$ we have the Frobenius norm $\|v\| = \sqrt{\sum_{j=1}^k \|v_j\|^2}$ and the inner product $\langle v,w\rangle = \tr(vw^T)$ for $v,w\in \RR^{n\times k}$. The $k\times k$ dimensional unit matrix is $I_k$ and $O(k) = \{R\in \RR^{k\times k}: R^TR = I_k\}$ is the orthogonal group. Moreover, e.g. \cite[p. 1]{James1976},
\begin{align*}
  O(k,m) = \Big\{v=(v_1,\ldots,v_{k}) \in \RR^{m\times k}: \langle v_i,v_k\rangle = \delta_{ik},\quad 1\leq i\leq j\leq k\Big\}\cong O(m)/O(m-k)\,,
\end{align*}
denotes the Stiefel manifold of orthonormal $k$-frames in $\RR^m$. For every such orthonormal $k$-frame we have a non-unique \emph{orthonormal complement} $\wv \in O(m-k,m)$ such that $(v,\wv) \in O(m)$.

In PNS and PNGS, for a top sphere $Q=\SSS^m$ a sequence of nested subspheres is sought for. Only in the scenario of PNGS it is required that each subsphere is a great subsphere. In general, every $j$-dimensional subsphere ($j=1,\ldots,m-1$) is the intersection of a $j+1$ dimensional affine subspace of $\RR^{m+1}$ with $\SSS^m$. Recall that every $j+1$ dimensional affine subspace
\begin{align*}
  A = \{x \in \RR^{m+1}: \langle x,v_k\rangle = \alpha_k,~k=1,\ldots,m-j\}
\end{align*}
is determined by a matrix
\begin{align*}
  v=(v_1,\ldots,v_{m-j})
\end{align*}
of $m+1 - (j+1) = m-j$ orthonormal column vectors 
that are orthogonal to $A$ and a vector of signed distances from the origin
\begin{align*}
  \alpha = (\alpha_1,\ldots,\alpha_{m-j})^T\,.
\end{align*}
In particular $\|\alpha\|<1$ ensures that $A$ intersects with $\SSS^m$ in a $j$ dimensional subsphere. Obviously, $A$ determines $v$ and $\alpha$ up to an action of $ O(m-j)$, i.e. with every $R \in O(m-j)$, $vR$ and $R^T \alpha$ determine the same $A$. 

In PNGS only great subspheres are allowed as intersections. Hence all affine spaces under consideration above pass through the origin, i.e. $\alpha = 0$ above. 

\textbf{The parameter space.}
In consequence we have that the family of $j$-dimensional subspheres ($j=1,\ldots,m-1)$ is given by the smooth manifold
\begin{align*}
  P_j = M_j/O(m-j) = \left\{[z]: z \in M_j\right\}\mbox{ with } [z] = \left\{z R: R\in O(m-j)\right\}
\end{align*}
where
\begin{align*}
  M_j &\:=\: \left\{\begin{array}{ll} O(m-j,m+1)\times \mathbb D^{m-j} = \left\{\twovec{v}{\alpha^T}: v\in O(m-j,m+1),\alpha \in \DD^{m-j}\right\}&\mbox{ for PNS}\\
  O(m-j,m+1)\times \{0\}&\mbox{ for PNGS}\end{array}\right.\, 
\end{align*}
with $\{0\}\subset \RR^{m-j}$ above, for compatibility, becoming clear in the considerations below.
Indeed, the smooth action from the right of the compact Lie group $O(m-j)$ on $M_j$ is free (for $v\in O(m-j,m+1)$ and $R\in O(m-j)$, $vR=v$ implies that $R=I_{m-j}$), giving rise to a smooth quotient manifold, e.g. \cite[Theorem 7.10]{Lee2013}, of dimension $(j+2)(m-j)$ for PNS and of dimension $(j+1)(m-j)$ for PNGS. Notably, in the latter case, $P_j$ is just a Grassmannian.

In the above setup, we had excluded the cases $j=0,m$. For $j=m$, for PNS and PNGS it is natural to set, as in Definition \ref{BNFD:def},
\begin{align*}
  P_m = \{\SSS^m\}\,.
\end{align*}
In case of $j=0$, the setup above would yield pairs of points (the intersections of suitable lines with $\SSS^m$ give topologically zero-dimensional spheres $\SSS^0$). In order to have a single nested mean as a zero dimensional descriptor only, for PNS and PNGS (in order to represent nestedness) we use the convention 
\begin{align*}
  P_0 = \SSS^m\mbox{ i.e. } M_0 = O(m,m+1)\times \SSS^{m-1}\mbox{ with }\SSS^{m-1} = \{\alpha \in \RR^{m}:\|\alpha\|=1\}\,.
\end{align*}

\textbf{Distance between subspheres of equal dimension.} For PNS and PNGS, on $M_j$ we have the \emph{extrinsic metric}
\begin{align*}
  \left(z,z'\right) \mapsto \| z - z'\| \,.
\end{align*}
Since the extrinsic metric is invariant under the action of $O(m-j)$ it gives rise to the well defined quotient metric, called the \emph{Ziezold metric}, cf. \cite{Z94}) on $P_j$ given by 
\begin{align}\label{extrinsic-metric:eq}
  d_j(p,p') = \min_{R\in O(m-j)} \| z-z'R\|
\end{align}
for arbitrary representatives $z,z' \in M_j$ of $p,p'\in P_j$, respectively, as the following Lemma teaches.

\begin{Lem}\label{lem:Ziezold-metric} The mapping $d_j: P_j \to P_j \to [0,\infty)$ satisfies the triangle inequality and it is definite, i.e. $d_j(p,p') = 0$ implies $p=p'$.
\end{Lem}

\begin{proof}
 Suppose that $z=(v^T,\alpha)^T,z'=({v'}^T,\alpha')^T,z''\in M_j$ are representatives of $p,p',p'' \in P_j$ with the property, w.l.o.g., that $d_j(p,p') = \|z-z'\|, d_j(p,p'') = \|z-z''\|$. Then the usual triangle inequality yields
\begin{align*}
   d_j(p',p') \leq \|z'-z''\| \leq \|z'-z\| +\|z-z'\| = d_j(p',p)+d_j(p,p'')\,.
\end{align*}

Moreover, we have 
\begin{align*}
\|z-z'R\|^2 = 2(m-j) + \|\alpha\|^2 + \|\alpha'\|^2 - 2\tr(v^Tv'R) - 2\tr(\alpha^TR^T\alpha')
\end{align*}
where $\tr(v^Tv'R) \leq m-j$ and  $\tr(\alpha^TR^T\alpha') \leq \|\alpha\|\,\|\alpha'\|$ with equality if and only if $v =v'R$ and $\alpha = c R^T\alpha'$ with $c\geq 0$, yielding that the above vanishes if and only if this is the case with $c=1$.
\end{proof}

\textbf{Backwards nesting.} In PNS, the space of subspheres $S_p \subset P_{j-1}$ within a given subsphere $p = [v^T,\alpha]^T\in P_j$ ($j=2,\ldots,m$) can be given the following structure.
\begin{align*}
  S_p &= \left\{\twoclass{v,v_{m-j+1}}{\alpha^T,\alpha_{m-j+1}}: v_{m-j+1} \in \SSS^{m},\alpha_{m-j+1}\in\RR\mbox{ with } v^T_{m-j+1}v=0, \alpha_{m-j+1}^2<1-\|\alpha\|^2\right\}\\
  &\subset P_{j-1}\,.
\end{align*}
Indeed, if $(R^Tv^T,R^T\alpha)^T$, $R\in O(m-j)$ is another representative for $p$, then $ v^T_{m-j+1}v=0$ if and only if $ v^T_{m-j+1}vR=0$.
In case of PNGS, the condition on entries of $\alpha$ above is simply $\alpha_{m-j+1}=0$ because $\alpha =0$.

According to our convention of $P_0=\SSS^m$, in case of $j=1$ the inequality above needs to be replaced with $\alpha_{m-j+1}^2=1-\|\alpha\|^2$ (for PNS) and with $\alpha_{m-j+1}^2=1$ for (PNGS).

\textbf{Projections.} For all $j=0,\ldots,m-1$ we have the intrinsic orthogonal projection onto $p = [v^T,\alpha]^T \in P_j$
\begin{align}\label{eq:blow-down}
  \pi_{\SSS^m,p^j} : \SSS^m \to p^j,\quad q\mapsto y=v\alpha + \sqrt{1-\|\alpha\|^2} ~\frac{(I_{m+1}-vv^T)q}{\|(I_{m+1}-vv^T)q\| }= v\alpha + \sqrt{1-\|\alpha\|^2}\,\frac{\wv\wv^Tq}{\|\wv^Tq\|}
\end{align}
which is independent of the representative $(v^T,\alpha)^T$ chosen and independent of the specific orthogonal complement $\wv \in O(j+1,m+1)$ of $v$ chosen; and it is well defined except for a set
\begin{align*}
  \left\{q\in \SSS^m : q = \sum_{i=1}^{m-j}\langle q,v_i\rangle v_i\right\}
\end{align*}
of spherical measure zero. Note that we have $\|\alpha\| =1$ for $j=0$ and hence the constant mapping $y\mapsto v\alpha$.

\textbf{Nested projections.} 
More generally, if $p^j,p^{j'}$ are from a family of backward nested subspheres $\SSS^m \supset p^{m-1}\supset \ldots \supset p^{1} \supset p^{0}$, $0\leq j'<j\leq m-1$, we may choose representatives $(v^T,\alpha)^T$ of $p^j$ and $({v'}^T,\alpha')^T$ of $p^{j'}$ such that
\begin{align*}
\arraycolsep=1.4pt
   \left.
  \begin{array}{rl@{\qquad}rl}v &= (v_1,\ldots,v_{m-j}),& v' &= (v_1,\ldots,v_{m-j'})\\
  \alpha^T &= (\alpha_1,\ldots,\alpha_{m-j}),& {\alpha'}^T &= (\alpha_1,\ldots,\alpha_{m-j'})\end{array}\right\}\,.
\end{align*}
 With (\ref{eq:blow-down}) and the arbitrary but fixed complement $\wv$ of $v$ chosen,
embed $p^j$ in $\RR^{j+1}$ (first a translation, possible in PNS, and then a blow up), depending on the specific choice of $\wv$, via
\begin{align}\label{eq:blow-up}
  g_{p^j,\SSS^j} : p^j \to \SSS^j,\quad y\mapsto z=\frac{\wv^Ty}{\|\wv^Ty\|}= \frac{\wv^Tq}{\|\wv^Tq\|}\mbox{ for $y=\pi_{\SSS^m,p^j}(q)$},~g_{p^j,\SSS^j}^{-1}(z)=v\alpha +\sqrt{1-\|\alpha\|^2}\wv z\,.
\end{align}
Now, $g_{p^j,\SSS^j}$ embeds $p^{j'}$ as a $j'$-dimensional (possibly small) subsphere $p'_{j'}= g_{p^j,\SSS^j}(p^{j'})\subset \SSS^j$ in $\RR^{j+1}$, given by $[w^T,\beta]^T = p'_{j'}$, for suitable  $(w^T,\beta)^T \in O(j-j',j+1) \times \DD^{j-j'}$, and there is an orthogonal complement $\ww \in O(j'+1,j+1)$ of $w$, such that
\begin{align}\label{proof-PNS-BNDS:eq2}
  v' = (v,\wv w),\quad \wv' = \wv \ww,\quad {\alpha'}^T=\left(\alpha^T,\beta^T\sqrt{1-\|\alpha\|^2}\right)
\mbox{ and  }(v',\wv') \in O(m+1)\,.
\end{align}
This gives the definition of the projection, independent of the specific orthogonal complements chosen,
\begin{align}\nonumber \label{eq:projections}
  \pi_{p^j,p^{j'}} = g_{p^{j},\SSS^{j}}^{-1} \,\circ\, \pi_{\SSS^j,p'_{j'}}\,\circ\, g_{p^j,\SSS^j}: ~p^j &\to~ p^{j'},\\
  \nonumber
  y&\mapsto~ v\alpha + \sqrt{1-\|\alpha\|^2}\,\wv\left( w\beta + \sqrt{1-\|\beta\|^2}\,\frac{\ww\ww^T\wv^Ty}{\|\ww^T\wv^Ty\|}\right)\\
  &=~   v'\alpha' + \sqrt{1-\|\alpha'\|^2} \frac{\wv'{{\widetilde{v'}}}^Ty}{\|{{\widetilde{v'}}}^Ty\|}\,.
\end{align}
Plugging in (\ref{eq:blow-down}) into the above equality and taking into account that $\wv^Ty/\|\wv^Ty\|=\wv^Tq/\|\wv^Tq\|$, cf. (\ref{eq:blow-up}),  yields at once the following proposition which asserts that projections along nested subspheres only depend on the final subsphere at which it ends. Recall that $\alpha =0$ ($=\beta$ if $j'> 0$) in case of PNGS. 

\begin{Prop}\label{PNS-BNDS:prop}
  With the above notation $\pi_{p^j,p^{j'}}\,\circ\, \pi_{\SSS^m,p^j} = \pi_{\SSS^m,p^{j'}}$.
\end{Prop}

\textbf{Distance between projected data and next subsphere.} Here, we compute the intrinsic geodesic distance $\rho_{p^j}(y,p^{j-1})$ between $y$ and $p^{j-1} =[{v'}^T,\alpha']^T\in S_{p^j}$ in a subsphere $p^j = [v^T,\alpha]^T \in P_j$ ($j=1,\ldots,m)$. Note that only in case of $p^j$ being a great subsphere, this distance agrees with the spherical distance $\arccos\big(y^T\, \pi_{p^j,p^{j-1}}(y)\big)$ in the top sphere $\SSS^m$. If $p^j$ is a proper subsphere ($0<\|\alpha\|<1$), assuming w.l.o.g. that $v'= (v,v_{m-j+1})$ and ${\alpha'}^T = (\alpha^T,\alpha_{m-j+1})^T$ we have 
\begin{align} \nonumber 
  \MoveEqLeft \rho_{p^j}(y,p^{j-1})\\ \nonumber
  &= \sqrt{1-\|\alpha\|^2}\arccos \left(\big(g_{p^j,\SSS^j}(y)\big)^T\, g_{p^j,\SSS^j}\circ \pi_{p^j,p^{j-1}}(y)\right)\\
  &= \sqrt{1-\|\alpha\|^2}\arccos\left(\frac{1}{1-\|\alpha\|^2} \left(y^Tv_{m-j+1}\alpha_{m-j+1} + \sqrt{1-\|\alpha'\|^2}y^T(I_{m+1}-v'{v'}^T)y
  \right)\right)\,. \label{rho-p-PNS:eq}
\end{align}
Indeed with $j'=j-1$ and the notation from (\ref{eq:blow-down}) to (\ref{eq:projections}) and orthogonal complements $\wv = (v_{m-j+1},\wv')$ of $v$ and $\wv'$ of $v'$, respectively,
we have 
\begin{align*}
  \big(g_{p^j,\SSS^j}(y)\big)^T\, g_{p^j,\SSS^j}\circ \pi_{p^j,p^{j'}}(y) 
  &=\frac{y^T\wv}{\|\wv^Ty\|} 
  \frac{\wv^T\left(v'\alpha' + \sqrt{1-\|\alpha'\|^2} \frac{\wv'{{\widetilde{v'}}}^Ty}{\|{{\widetilde{v'}}}^Ty\|}\right)}{\left\|\wv^T\left(v'\alpha' + \sqrt{1-\|\alpha'\|^2} \frac{\wv'{{\widetilde{v'}}}^Ty}{\|{{\widetilde{v'}}}^Ty\|}\right)\right\|}
\end{align*}
since $\|\wv^T y\| = \sqrt{1-\|\alpha\|^2}$ due to (\ref{eq:blow-down}), $y^T \wv \wv^T v' \alpha' = y^Tv_{m-j+1}\alpha_{m-j+1}$ and $y^T \wv \wv^T \wv'{{\widetilde{v'}}}^Ty = y^T\wv'{{\widetilde{v'}}}^Ty$.

\textbf{Optimal positioning.} On $M_j$ we have the extrinsic metric due to its embedding in $\RR^{(m+1)\times(m-j)}\times \RR^{m-j}$, cf. (\ref{extrinsic-metric:eq}), and its thus induced Riemannian metric. Obviously $O(m-j)$ acts on $M_j$ isometrically w.r.t. the extrinsic metric. It also acts isometrically w.r.t. the Riemannian metric because the action also preserves geodesics on $O(m-j,m+1)$, e.g. \cite[p. 309]{EAS98}. In consequence, for both metrics, we say for given $z,z'\in M_j$ that $R\in O(m-j)$ puts $z$ i.o.p. (\emph{in optimal position}) to $z'$,  if
\begin{align*}
  \| z'-zR\| =\min_{R'\in O(m-j)} \| z'-zR'\|\,.
\end{align*}
%

\begin{Lem}\label{lem:smooth-op} 
  Let $j \in \{0,\ldots,m-1\}$ and $z,z' \in M_j$ be sufficiently close. Then, there is a unique $R_{z,z'} \in O(m-j)$ such that
  \begin{enumerate}[(i)]
    \item $R_{z,z'}$ puts $z$ uniquely i.o.p. $zR_{z,z'}$ to $z'$,
    \item if ${z'}^T z$ is symmetric then $R_{z,z'} = I_{m-j}$.
  \end{enumerate}
\end{Lem}
\begin{proof} 
  In order that $R\in O(m-j)$ minimizes the r.h.s of (\ref{extrinsic-metric:eq}), given by
 $ 
    \|z' - zR\|^2 = \|z'\|^2 + \|z\|^2 -2\,\tr(AR)\,,
 $ 
  with $A= {z'}^Tz$, it maximizes $\tr(AR)$. This is so if $R = QS^T$ for a singular value decomposition (svd) $S\Lambda Q^T=A$ stemming from a spectral decomposition $AA^T = S\Lambda^2S^T$. 
  Since for $z=z'=({v'}^T,\alpha')^T$ we have that $A= I_{m-j} + \alpha'{\alpha'}^T$ is of full rank for all $\alpha'\in \RR^{m-j}$, there is locally a full rank svd, which is unique up to $S\mapsto SL$ and $Q = A^TS\Lambda^{-1} \mapsto A^TSL\Lambda^{-1}$ for any $L \in O(m-j)$ with $L\Lambda L^T = \Lambda$. However, $R_{z,z'}=A^TS\Lambda^{-1} S^T$ is unique under actions of such $L$, yielding (i).
  (ii): The symmetry $S\Lambda Q^T=A=A^T=Q\Lambda S^T$ allows to choose $Q=S$, i.e. $R_{z,z'}= I_{m-j}$.
  \end{proof}

This Lemma has the following immediate consequence. 
\begin{Cor}
  If $z,z' \in M_j$ are sufficiently close then
  \begin{align*}
    z,z'\mbox{ are i.o.p. }\Leftrightarrow z^Tz' - {z'}^T z=0\,.
  \end{align*}
\end{Cor}

Thus, in further consequence, for every $z' \in M_j$ there is $\epsilon >0$ such that 
\begin{align}\label{eq:chart-Pj}
  U_j^{z'} &\::=\: \left\{z\in M_j: z^Tz' - {z'}^T z=0, \|z-z'\|<\epsilon\right\}
\end{align}
is a smooth manifold and every $p'\in P_j$ has a neighborhood $V$ with a smooth diffeomorphism
\begin{align}
  \label{eq:1chart-Pj} V\to U_j^{z'}, \qquad p \mapsto z = \widetilde{z}R_{\widetilde{z},z'}
\end{align}
where $z'\in M_j$ is a fixed representative of $p'$ and $\widetilde{z}\in M_j$ is an arbitrary representative of $p$, and $z$ is a representative i.o. p to $z'$. This follows from the fact that locally  a point is i.o.p to $z'$ if and only if it can be reached by a \emph{horizontal} geodesic from $z'$, and from the fact that all geodesics in $P_j$ through $p'$ lift to horizontal geodesics in $M_j$ through $z'$ and all horizontal geodesics in $M_j$ project to geodesics in $P_j$. For a detailed discussion e.g. \cite{HHM07}.

\textbf{Representing spaces of nested subspheres: Factoring charts.} 
For $j \in \{1,\ldots,m-1\}$ recall
\begin{align*}
  T_{j,1} = \{(p,s): p \in P_j, s\in S_p\}\,,
\end{align*}
and that if $y=(v^T,\alpha)^T\in M_j$ is a representative of $p \in P_j$, then every $s\in S_p$ is represented by some
\begin{align*}
  z=\twovec{v,v_{m-j+1}}{\alpha^T,\alpha_{m-j+1}}\in M_{j-1}
\end{align*}
with suitable choices $v_{m-j+1}\in \SSS^{m}$, $v^Tv_{m-j+1}=0$ and $\alpha_{m-j+1}^2 < 1- \|\alpha\|^2$ (for PNS, $j>1$), $\alpha_{m-j+1} =0=\|\alpha\|$ (for PNGS, $j>1$) and $\alpha_{m-j+1}^2 = 1- \|\alpha\|^2$ (for $j=1$). Vice versa, from every representative $z=(w^T,\beta)^T\in M_{j-1}$ of $s\in P_{j-1}$, all of the $p\in P_j$ with $s\in S_p$ are determined by choosing $m-j$ orthonormal vectors from the column span of $w$ (the columns of $wB$ below) with suitable $m-j$ distances (given by the vector $\beta^TB$ below), i.e. every such $p$ has a representation
\begin{align*}
  p = [zB]
\end{align*}
as $B$ ranges over $O(m-j,m-j+1)$. In case of PNGS and $j=1$, for compatibility we set $(0,\ldots,0,1)B= (0,\ldots,0)$.
Of course, different $B$ may give same the $p$. More precisely, here is the central observation.

Given $s = [z]\in P^{j-1}$, every $p\in P^j$ with $s\in S_p$ is uniquely determined as $p=[zB]$ by choice of $b\in \SSS^{m-j}$, where $B$ is an arbitrary orthonormal complement of $b$ in $O(m-j+1)$, i.e. $(B,b) \in O(m-j+1)$. This means that $\SSS^{m-j}$ uniquely parametrizes $\{p \in P^{j}: s\in S_p\}$. The reason is the well known fact that the Grassmannian $O(m-j,m-j+1)/SO(m-j) = SO(m-j+1)/SO(m-j)$ is diffeomorphic with $\SSS^{m-j}$. 
Again, in case of PNGS and $j=1$, for compatibility we set $(0,\ldots,0,1)b= 0$.

This observation gives rise to  \emph{factoring charts} as introduced in Definition \ref{def:factoring}. Here and below in Theorem \ref{th:several-nested-subspheres}, $U_{j-1}^{z'}$ assumes the role of $U$ from Definition \ref{def:factoring}.

\begin{Lem}\label{lem:two-nested-subspheres}
  Every $(p',s')\in T_{j,1}$ has a neighborhood $W\subset T_{j,1}$ and a smooth diffeomorphism
  \begin{align*}
    \phi : W \to U_{j-1}^{z'}\times V,~~(p,s) \mapsto (z,b)\,
  \end{align*}
  where $U_{j-1}^{z'}$ is from (\ref{eq:chart-Pj}), $V\subset \SSS^{m-j}$ is a neighborhood of $(0,\ldots,0,1)^T\in \RR^{m-j+1}$ and
  \begin{align*}
    z' = \twovec{{v'},v'_{m-j+1}}{{\alpha'}^T,\alpha_{m-j+1}}\mbox{ and } \twovec{v'}{{\alpha'}^T}
  \end{align*}
  are representatives of $s'$ and $p'$ respectively.
  In particular,
  \begin{align*}
    s = [z]\mbox{ and } p = [zB]
  \end{align*}
  with arbitrary orthonormal complement $B$ of $b$.
\end{Lem}

\begin{proof}
  Let $s = [\tilde{z}]$ for some $\tilde{z}\in M_j$ be sufficiently close to $s'$. Then, according to Lemma \ref{lem:smooth-op}, $s$ has a unique representative $z=\tilde{z}R_{\tilde{z},z'}$ in optimal position to $z'$, varying smoothly in $\tilde{z}$, hence smoothly in $s$. Moreover, as elaborated above, every $p\in P_j$ with $s\in S_p$ depends uniquely upon the choice of $b\in \SSS^{m-j}$ and any orthonormal complement $B$, i.e. $p = [zB]$. Although $-b$ determines the same $m-j$ dimensional subspace, a one-to-one mapping is obtained choosing the neighborhood $V$ of $(0,\ldots,0,1)$ suitably small not containing antipodals.
\end{proof}

The following is a straightforward generalization to spaces of sequences of several nested subspheres.
\begin{Th}[Factoring Charts]\label{th:several-nested-subspheres}
  Let $j \in\{1,\ldots,m-1\}$ and $k\in \{1,\ldots,j\}$. Then
  every $({p'}^j,\ldots, {p'}^{j-k})\in T_{j,k}$ has a neighborhood $W\subset T_{j,k}$ and a smooth diffeomorphism
  \begin{align*}
    \phi : W \to U_{j-k}^{z'}\times \prod_{r=1}^kV_r,~~(p^j,\ldots,p^{j-k}) \mapsto (z,b_1,\ldots,b_k)\,
  \end{align*}
  where $U_{j-k}^{z'}$ is from (\ref{eq:chart-Pj}), each $V_r\subset \SSS^{m-j+r-1}$ is a neighborhood of $(0,\ldots,0,1)^T\in \RR^{m-j+r}$,
  \begin{align*}
    z' = \twovec{v'_1,\ldots, v'_{m-j+k}}{{\alpha'_1},\ldots,{\alpha'_{m-j+k}}}\mbox{ and } \twovec{v'_1,\ldots, v'_{m-j+r}}{\alpha'_1,\ldots,\alpha'_{m-j+r}}
  \end{align*}
  are representatives of ${p'}^{j-k}$ and ${p'}^{j-r}$ respectively ($r=1,\ldots,k$). In particular,
  \begin{align*}
    p^{j-k} = [z]\mbox{ and } p^{j-r+1} = [zB_k\cdots B_{r}]
  \end{align*}
  with arbitrary orthonormal complements $B_r$ in $O(m-j+r)$ of $b_r\in V_r$ ($r=1,\ldots,k$).
\end{Th}

\textbf{A joint representation.} Here we assume that $f'=({p'}^{m-1},\ldots, {p'}^{0})$ and $f=({p}^{m-1},\ldots, {p}^{0})$ are two sufficiently close families of backward nested spheres with representatives
\begin{align*}
  {p'}^{j} = \twoclass{v'_1,\ldots,v'_{m-j} }{\alpha'_1,\ldots,\alpha'_{m-j}},\quad {p}^{j} = \twoclass{v_1,\ldots,v_{m-j} }{\alpha_1,\ldots,\alpha_{m-j}},\quad j=0,\ldots,m
\end{align*}
such that
\begin{align*}
  \twovec{v'_1,\ldots,v'_{m} }{\alpha'_1,\ldots,\alpha'_{m}},\quad \twovec{v_1,\ldots,v_{m} }{\alpha_1,\ldots,\alpha_{m}}
\end{align*}
are in o. $O(m)$-p. Notably, there are then uniquely determined $SO(m+1)$ complements $v'_{m+1}$ and $v_{m+1}$ respectively. For a sufficiently small neighborhood $W\subset T_{m-1,m-1}$ of $f'$ and neighborhoods $V_r\subset \SSS^r$ of $(0,\ldots,0,1)^T$ in $\SSS^{r}$ ($r=1,\ldots,k$) we have then a smooth mapping
\begin{align*}
  W \to \prod_{r=1}^m V_r,\quad f \mapsto (b_m,\ldots,b_1)
\end{align*}
determined by the following algorithm (here is only the PNGS version). For arbitrary $v\in \RR^{\dim(v)}$, $\Pi_{v^\perp}: x \mapsto (I_{\dim(v)}- vv^T)x$ denotes the orthogonal projection to the complement of $v$.
\begin{itemize}
  \item $b_m \in \SSS^{m}$ is the unique element such that $v_m = (v'_1,\ldots,v'_{m+1})b_m$.
  \item $b_{m-1} \in \SSS^{m-1}$ is the unique element such that $v_{m-1} = ({v'}^{(1)}_1,\ldots,{v'}^{(1)}_m)b_{m-1}$. Here, $({v'}^{(1)}_1,\ldots,{v'}^{(1)}_m)$ is obtained from $(\Pi_{v_m^\perp}v'_1, \ldots, \Pi_{v_m^\perp}v'_{m+1})$ by removing one column such that $v_1,\ldots,v_{m-1}$ is in the span of the rest (usually $\Pi_{v_m^\perp}v'_m$).
  \item[] $\vdots$
  \item $b_{1} \in \SSS^{1}$ is the unique element such that $v_{1} = ({v'}^{(m-1)}_1,{v'}^{(m-1)}_2)b_{1}$. Here, $({v'}^{(m-1)}_1,{v'}^{(m-1)}_2)$ is obtained from $(\Pi_{v_2^\perp}{v'}^{(m-2)}_1, \Pi_{v_2^\perp}{v'}^{(m-2)}_{2}, \Pi_{v_2^\perp}{v'}^{(m-2)}_{3})$ by removing one column such that $v_1$ is in the span of the other two. 
\end{itemize}

\textbf{Local charts.} For $P_j$ and $T_{j,k}$, we have derived in \eqref{eq:chart-Pj} and Theorem \ref{th:several-nested-subspheres}, respectively, local smooth diffeomorphic representations in a Euclidean space of suitable dimension, $r$, say, that can be generically written as
\begin{align*}
  U = \{x\in V : \Phi(x) = 0\}
\end{align*}
with a neighborhood $V\subset \RR^r$ of some $x'\in \RR^r$ and a smooth mapping $\Phi: V \to \RR^s$, $r,s\in \NN$, $r> s$ with full rank derivative $d\Phi_{x'}$ at $x'$. Here, the corresponding matrices are viewed as vectorized by stacking their columns on top of one another. With $\Pi_{\cR d\Phi_{x'}}$, the orthogonal projection to the row space of the $s\times r$ matrix of the derivative at $x'\in V$ and orthogonal $r$-vectors $y_1,\ldots,y_t$, spanning the kernel of $d\Phi_{x'}$ of dimension $t=r-s$, obtain the local chart
\begin{align}
  \label{eq:local-chart} U \to \RR^{t}, \qquad x \mapsto (y_1,\ldots,y_t) (I_r - \Pi_{\cR d\Phi_{x'}}) (x-x')\,.
\end{align}

\subsection{Intrinsic Mean on a First Geodesic Principal Component for Kendall's Shape Spaces}\label{scn:Kendalls-shape-spaces}

Another application is given by the intrinsic mean on a geodesic principal component on a quotient space $Q=M/G$ due to an isometric action of a Lie group $G$ on a Riemannian manifold $M$. We treat here the prominent application of Kendall's shape spaces
\begin{align*}
  Q=\Sigma_m^k = \SSS^{m(k-1)-1}/SO(m)=\{[x]:x\in \SSS^{m(k-1)-1}\},\quad ,[x]=\{gx: g\in SO(m)\}
\end{align*}
where the space $\SSS^{m(k-1)-1}$ of unit size $m\times (k-1)$ matrices -- corresponding to normed and centered configuration of $k$ landmarks in $m$-dimensional Euclidean space -- is taken modulo the group of rotations in $m$-dimensional space. This space models all $m$-dimensional configurations of $k$ landmarks, not all coinciding, modulo similarity transformations, cf. \cite{DM98}. For $m\geq 3$ this space is no longer a manifold but decomposes into strata of manifolds of different dimensions, cf. \cite{KBCL99}. Geodesics on the unit sphere $\SSS^{m(k-1)-1}$, i.e. great circles, orthogonal to the the orbits $[x]$, called \emph{horizontal great circles}, project to geodesics in $Q=\Sigma_m^k$ such that the space $P_1$ of geodesics of $Q$ can be given the quotient structure of a stratified space
\begin{align*}
  P_1 = O^H\big(2,m\times (k-1)\big)/ \left( SO(m)\times O(2) \right)
\end{align*}
with the horizontal Stiefel manifold
\begin{align*}
  O^H\big(2,m\times (k-1)\big) = \{&(x,v) \in \RR^{m\times (k-1)}\times \RR^{m\times (k-1)}: \\
  &\tr(x^Tv)=0=\tr(x^Tx)-1 = \tr(v^Tv)-1, ~xv^T=vx^T\}
\end{align*}
and the orbits
\begin{align*}
  [[x,v]] = \{ (gx,gv) h: g\in SO(m), h \in O(2)\}\,
\end{align*}
for $(x,v) \in O^H\big(2,m\times (k-1)\big)$, cf. \cite[Theorem 5.2]{HHM07}. The action of $SO(m)$ is not free for $m\geq 3$, giving rise to a non-trivial stratified structure. As before in Section \ref{IMo1GPC:scn}, we set
\begin{align*}
  P_2=\{Q\},~P_0 = Q\,.
\end{align*}

Having geodesics, orthogonal projections $\pi_{Q,p} : Q\to p, p\in P_1$ can be defined, which are unique outside a set of intrinsic measure zero, cf. \cite[Theorem 2.6]{HHM07}. The geodesic distance $d$ on $Q$ gives rise to the geodesic distance $\rho_Q=\rho, \rho(q,p) = \min_{q'\in p}d(q,q')$ between a datum $q$ and a geodesic $p$. Similarly the induced intrinsic distance on $p$ gives rise to $\rho_p: p \times S_p \to [0,\infty)$ where $S_p = \{s\in p:s\in Q\}$ can be identified with $p$. For $d_1$ and $d_0$, for simplicity, not the canonical intrinsic distances on $P_1$ and $P_0=Q$ but quotient distances due to the embedding of $O^H\big(2,m\times (k-1)\big)\hookrightarrow \RR^{m\times (k-1)}\times \RR^{m\times (k-1)}$ and $\SSS^{m(k-1)-1}\hookrightarrow \RR^{m\times (k-1)}$, respectively, can be used, called the Ziezold distances, cf. \cite{H_ziez_geod_10}.

The generalized Fr\'echet mean corresponding to $\rho$ is the first geodesic principal component (GPC) $p$, the generalized Fr\'echet mean corresponding to $\rho_p$ is the the intrinsic mean on the first GPC, cf. \cite{HHM07,H_ziez_geod_10}. The latter is again a nested mean. With the \emph{horizontal projective bundle} over the unit sphere
\begin{align*}
  P^H \SSS^{m(k-1)-1} &= \mathop{\bigcup}_{x\in \SSS^{m(k-1)-1}}\{x\}\times P^H_x \SSS^{m(k-1)-1}\\
  P^H_x \SSS^{m(k-1)-1} &= \{\{v,-v\}: v\in T_x \SSS^{m(k-1)-1}, xv^T = vx^T, \tr(v^Tv) = 1\}
\end{align*}
we have the space
\begin{align*}
  T_{1,1} = \{(p,s): p \in P_1, s\in S_p\} \cong\{([[x,v]],[x]): x\in \SSS^{m(k-1)-1}, \{v,-v\} \in P^H_x \SSS^{m(k-1)-1}\}\,.
\end{align*}

The Principal Orbit Theorem (e.g. \cite[p. 199]{Bre72} states in particular, that $P_0=Q$ and $P_1$ have open and dense subsets $Q^*\subset Q, P_1^*\subset P_1$ that are manifolds, in our case smooth manifolds. This gives rise to the manifold
\begin{align*}
  T_{1,1}^* = \{(p,s) \in T_{1,1}: p\in P_1^*, s \in Q^*\}\,
\end{align*}
with local smooth coordinates near $(p',s') \in W\subset T_{1,1}^*$
\begin{align*}
  W \to P^H\SSS^{m(k-1)-1},~~ (p,s) \mapsto (x,\{v,-v\})
\end{align*}
where $x \in s$ and $(x,v)$ is a representative of $p$ i.o.p. to $(x',v')$, under the action of $SO(m)\times O(2)$, with $\tr(v^Tv') >0$. Here $(x',v')$ is an arbitrary but fixed representative of $p'$ such that $x'$ is a representative of $s'$. Along the lines of Lemma \ref{lem:smooth-op} one can show that optimal positioning is unique if $x'x^T + v'v^T$ has rank $\geq m-1$, which may be assumed for most realistic data scenarios. 


Arguing as in Section \ref{IMo1GPC:scn}, with every local trivialization of the horizontal bundle
\begin{align*}
  HS^{m(k-1)-1} = \{(x,v) \in \SSS^{m(k-1)-1}\times \RR^{m(k-1)}: \tr(x^Tv) =0\mbox{ and } xv^T=vx^T\}
\end{align*}
comes a factoring chart.

\section{Assumptions for the Main Results}\label{assumptions:scn}
In this section we are back in the general scenario described in Section \ref{general_framework:scn}. We develop a set of assumptions necessary for the general results on asymptotic consistency and asymptotic normality in Section \ref{asymptotics:scn}. We then show that they are fulfilled in case of PNS/PNGS and the IMo1GPC of Kendall's shape spaces. 
\subsection{Assumptions for Strong Consistency}


For the following assumptions suppose that $j \in\{ 1,\ldots,m-1\}$.
\begin{As}\label{Ass-meas-exists:as}
  For a random element $X$ in $Q$, assume that $\E[\rho_{p^{j}}(\pi_{f}\circ X,s)^2]<\infty$ for all BNFDs $f$ ending at $p^j$, $s\in S_{p^j}$.
\end{As}

In order to measure a difference between $s\in S_p$ and $s' \in S_{p'}$ for $p,p' \in P_j$ define the orthogonal projection of $s\in S_p$ onto $S_{p'}$ as
\begin{align*}
  S^s_{p'} =\argmin_{s'\in S_{p'}} \limits d_{j-1}(s,s')\,.
\end{align*}
In case of PNS this is illustrated in Figure \ref{pns_illustration:fig} (a).

\begin{As}\label{projection-unique:as}
  For every $s\in S_p$ there is $\delta>0$ such that
  \begin{align*}
    |S_{p'}^s| = 1
  \end{align*}
  whenever $p,p'\in P_j$ with $d_j(p,p')<\delta$.
\end{As}

\begin{figure}[h!]
  \centering
  \subcaptionbox{Nested projection}[0.48\textwidth]{\includegraphics[width = 0.45\textwidth, clip=true, trim=3.5cm 1cm 3cm 1cm]{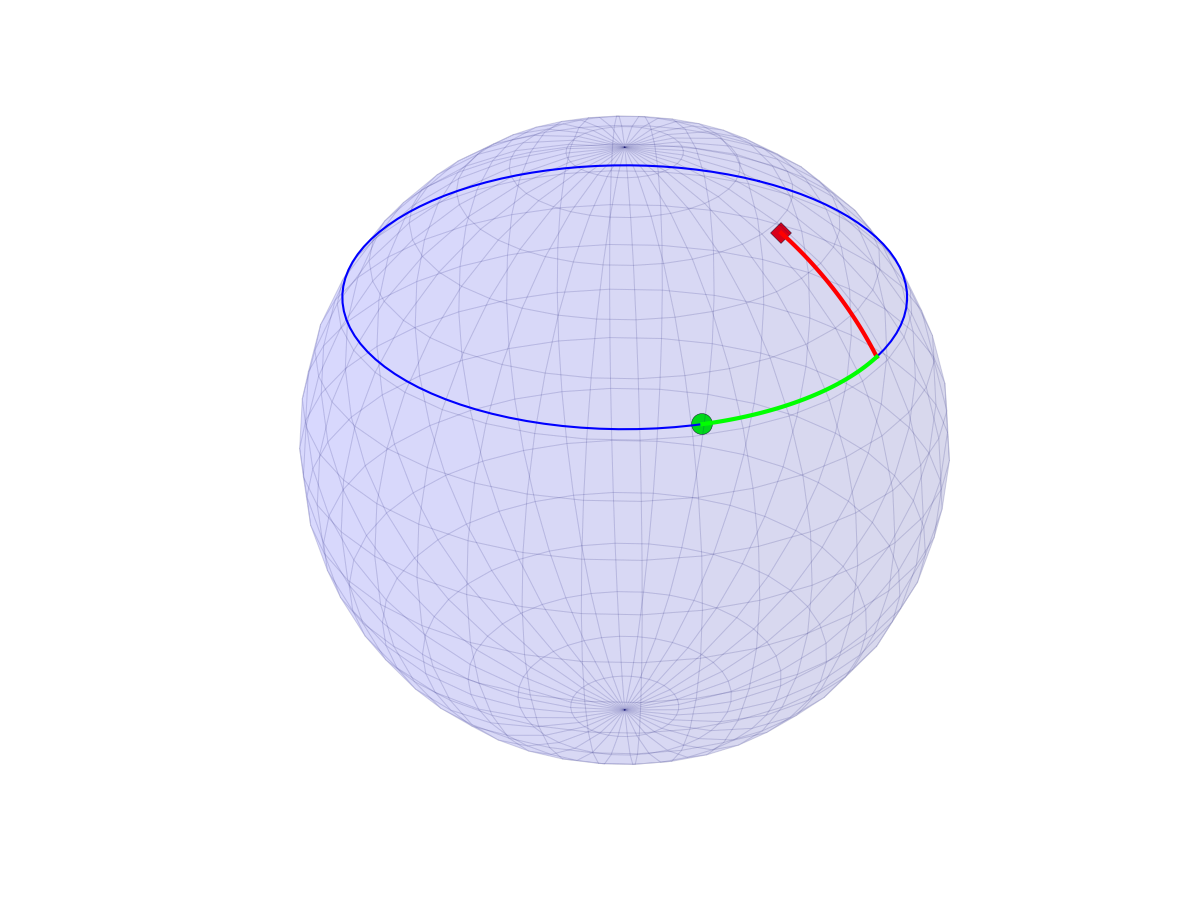}}
  \subcaptionbox{Projection of descriptors}[0.48\textwidth]{\includegraphics[width = 0.45\textwidth, clip=true, trim=3.5cm 1cm 3cm 1cm]{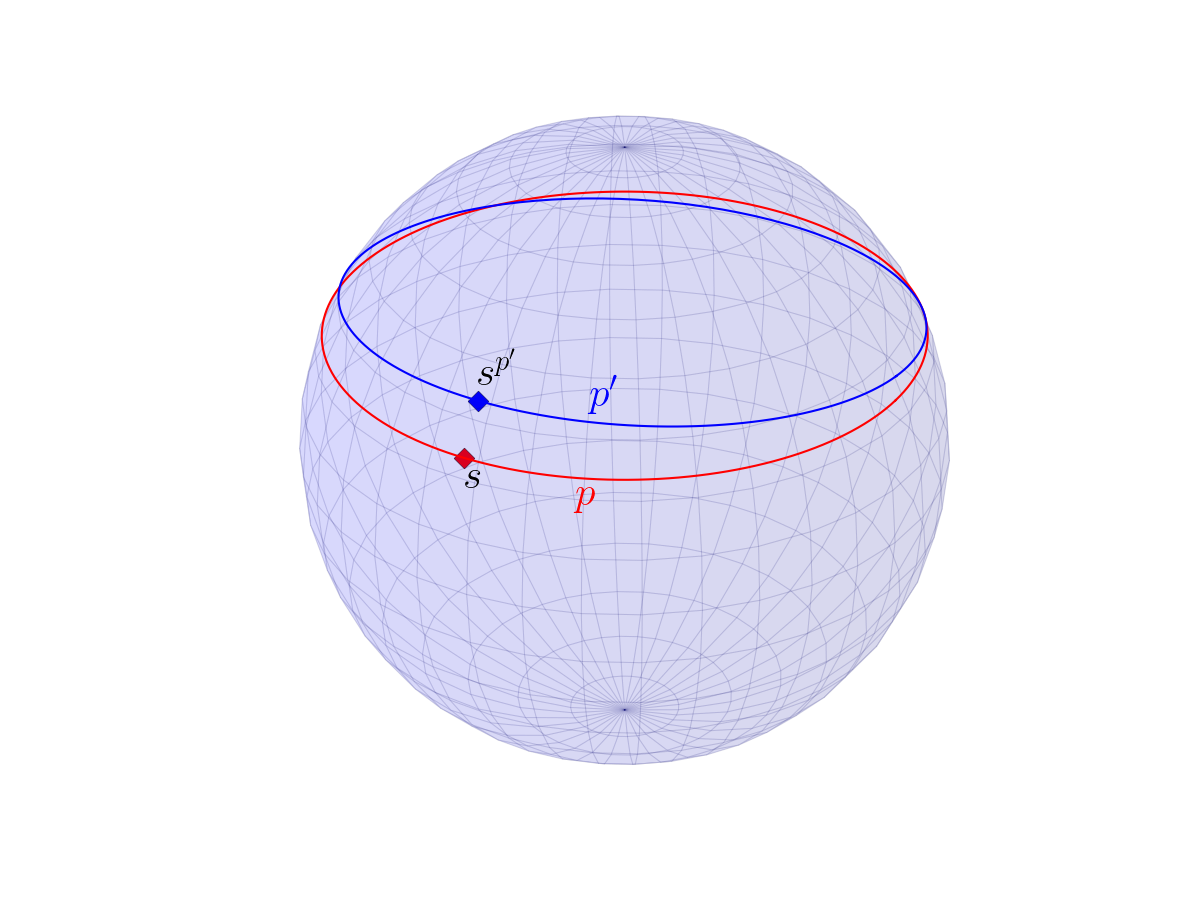}}
  \caption{\it PNS illustration. Left: Projection of $X$ (red) in $Q = \SSS^2$ onto small circle $p$ (blue) and further onto $s$ (green). Right: Projection $s^{p'}$ (blue) onto $S_{p'}$ (which is $p'$ in this case) of $s$ (red) on $S_{p}$ (which is $p$ in this case).\label{pns_illustration:fig}} 
\end{figure}

For $s\in S_p$ and $p,p'\in P_j$ sufficiently close let $s^{p'} \in S_{p'}^s$ be the unique element. Note that in general 
\begin{align*}
  (s^{p'})^p \neq s\,.
\end{align*}
In the following assumption, however, we will require that they will uniformly not differ too much if $p$ is close to $p'$. Also, we require that $s^{p'}$ and $s$ be close.

\begin{As}\label{projection-close:as}
  For $\epsilon >0$ there is $\delta>0$ such that
  \begin{align*}
    d_{j-1}(s^{p'},s) <\epsilon\mbox{ and }d_{j-1}\big((s^{p'})^p,s\big)<\epsilon\quad\forall s \in S_p
  \end{align*}
  whenever $p,p'\in P_j$ with $d_j(p,p')<\delta$.
\end{As}

We will also require the following assumption, which, in conjunction with Assumption \ref{projection-close:as}, is a consequence of the triangle inequality, if $d_{j-1}$ is a metric.

\begin{As}\label{almost-triangle:as}
  Suppose that $d_j(p_n,p) \to 0$ and $d_{j-1}(s_n,s)\to 0$ with $p,p_n \in P_j$ and $s\in S_p, s_n\in S_{p_n}$. Then also
  \begin{align*}
    d_{j-1}(s_n,s^{p_n}) \to 0
  \end{align*}
\end{As}

Moreover, we require uniformity and coercivity in the following senses.
\begin{As}\label{uniform-link:as}
  For all $\epsilon >0$ there are $\delta_1,\delta_2>0$ such that
  \begin{align*}
    \Big|\rho_{p}\big(\pi_f(q),s\big) - \rho_{{p'}}\big(\pi_{f'}(q),s'\big)\Big| <\epsilon \quad \forall q\in Q
  \end{align*}
  for all BNFDs $f,f'\in T_{m-1,m-j-1}$ ending in $p,p'\in P_j$, respectively, with $d(f,f')<\delta_1$ and $s\in S_p, s'\in S_{p'}$ with $d_{j-1}(s,s') < \delta_2$.
\end{As}

\begin{As}\label{coercivity:as}
  If $p_n,p \in P_{j}$ and $s_n\in S_{p_n}, s\in S_{p}$ with $d_{j-1}(s_n,s) \to \infty$, then for every $C>0$ we have that
  \begin{align*}
    \rho_{p_n}(\pi_{f_n}q,s_n) \to \infty
  \end{align*}
  for every $q\in Q$ with $\rho_{p}(\pi_fq,s) < C$ and BNFDs $f,f_n\in T_{m-1,m-j-1}$ ending at $p,p_n$ respectively.
\end{As}

\begin{Rm}\label{compact:rm}
 Due to continuity, Assumptions \ref{Ass-meas-exists:as} and \ref{uniform-link:as} hold if $Q$ is compact and Assumption \ref{coercivity:as} if each $P_j$ is compact. 
\end{Rm}

\begin{Th}
  Assumptions \ref{Ass-meas-exists:as} -- \ref{coercivity:as} hold for PNS and PNGS for all $j = 1,\ldots,m-1$. Moreover, each $P_j$ is $d_j$-Heine Borel for $j=0,\ldots,m$.
\end{Th}

\begin{proof}
  Recall from Proposition \ref{PNS-BNDS:prop} that $\pi_f$ only depends on the final descriptor at which $f$ ends. For PNS and PNGS we use the notation introduced in Section \ref{setup:scn} and show first Assumption \ref{projection-unique:as}. 
  Let $(v^T,\alpha)^T$ and $({v'}^T,{\alpha'})^T$ be representatives of $p,p'\in P_j$ in optimal position and $ s = [w^T,\beta]^T\in S_p$ with $w=(v,v_{m-j+1}) \in O(m-j+1,m+1)$ and $\beta^T = (\alpha^T,\alpha_{m-j+1})$. Moreover consider a candidate element
  \begin{align*}
  s^* = \twoclass{v',x}{{\alpha'}^T,y} \in S_{p'},\quad x\in \SSS^{m+1}, x^Tv' =0,~\left\{\begin{array}{ll}y\in (-1,1), y^2 + \|\alpha'\|^2 < 1&, \mbox{ for PNS}, j>1\\
	  y=0&, \mbox{ for PNGS}, j>1\\
    y\in \{-1,1\}, y^2 + \|\alpha'\|^2 = 1&, \mbox{ for PNS/PNGS}, j=1\\
    \end{array}\right.
  \end{align*} 
  and its squared distance to $s$, 
  \begin{align*}
  d^2_j(s^*,s) &= \mathop{\min}_{R \in O(m-j)} \left\| \twovec{w}{\beta^T} R -  \twovec{v',x}{{\alpha'}^T,y}\right\|^2
  \\
  &=2(m-j+1) + \|\beta\|^2 + \|\alpha'\|^2 + \|y\|^2 - 2 \mathop{\max}_{R \in O(m-j)}\tr(H R)
  \end{align*}
  with
  \begin{align*}
  H = \left(\begin{array}{cc} v^T{v'} + \alpha{\alpha'}^T & v^Tx + \alpha y\\ v_{m-j+1}^Tv' + \alpha_{m-j+1} {\alpha'}^T& v_{m-j+1}^Tx + \alpha_{m-j+1} y\end{array}\right)\,.
  \end{align*}

  With the continuous function 
  $F(v',\alpha') = \min_{x,y} d^2_j(s^*,s)$,
    we have $F(v,\alpha) = 0$, which, by Lemma \ref{lem:Ziezold-metric}, is uniquely assumed for $x=v_{m-j+1}$ and $y=\alpha_{m-j+1}$.
  Due to continuity, for $(v',\alpha')$ sufficiently close to $(v,\alpha)$ all minimizers $(x,y)$ of $F(v',\alpha')$ are in a neighborhood of $(v_{m-j+1},\alpha_{m-j+1})$ and there, $H$ is full rank and hence, arguing as in the proof of (i) of Lemma \ref{lem:smooth-op}, the extremal $R\in O(m-j)$ is uniquely determined. Let us write $R=(B,b)$ to obtain
   \begin{align*}
  d^2_j(s^*,s) &=\min_{\begin{array}{c} B \in O(m-j,m-j+1)\\ 0=B^Tb,b\in \SSS^{m-j}\end{array}} \left( \left\| \twovec{w}{\beta^T} B -  \twovec{v'}{{\alpha'}^T}\right\|^2 +  \left\| \twovec{w}{\beta^T}b -  \twovec{x}{y}\right\|^2\right)\,.
  \end{align*}
  As $b$ is unique, minimizing the above for $x$ is equivalent to minimizing $\|wb-x\|^2$ under the constraining conditions $\|x\|^2=1$ and ${v'}^Tx=0$ which yields the necessary equation
  $$ x - wb + v' \lambda + \mu x =0, \mbox{ with suitable Lagrange multipliers }\lambda \in \RR^{m-j},\mu \in \RR\,.$$
  Right multiplication with ${v'}^T$ yields at once $\lambda = {v'}^Twb$ such that $x$ is uniquely determined by
  $$ x = \frac{(I_{m+1}-v'{v'}^T)wb}{\|(I_{m+1}-v'{v'}^T)wb\|}\,.$$
  Indeed, $x$ is well defined because it is in a neighborhood of $v_{m-j+1}$ and hence $b$ in a neighborhood of $(0,\ldots,0,1)^T$.
  More simply, without Lagrange minimization, we obtain $y = \beta^Tb$. 
  
  For $p,p'$ sufficiently close, this yields a unique $s^*=s^{p'}$ minimizing $d_j(s^*,s)$ over $s^* \in S_{p'}$, yielding at once Assumptions \ref{projection-unique:as}, \ref{projection-close:as} and \ref{almost-triangle:as} (because $d_j$ is a metric).

  Due to Remark \ref{compact:rm}, Assumptions \ref{Ass-meas-exists:as}, \ref{uniform-link:as} and 
  \ref{coercivity:as} hold. 
  Because each $P_j$ is a finite dimensional manifold and $d_j$ is a topologically compatible metric, $P_j$ is $d_j$-Heine Borel.
\end{proof}
\begin{Th}
  For IMo1GPCs on Kendall's shape spaces $Q=\Sigma_m^k$, $0<m<k$, Assumptions \ref{Ass-meas-exists:as} -- \ref{coercivity:as} hold for $j=1$. Moreover, $P_j$ is $d^j$ Heine-Borel for $j=0,1$.
\end{Th}

\begin{proof}
  Assumption \ref{projection-unique:as} follows at once from the compactness of $Q$, hence the geodesics $p\subset Q$ are also compact and the proof of \cite[Theorem A.5]{HHM07}, as there, in Claim II, a neighborhood of a geodesic $p$ is constructed, restricted to which the orthogonal projection $\pi_p$ is well defined and continuous in $p$. Compactness and continuity also imply Assumptions \ref{projection-close:as} and \ref{almost-triangle:as}. Assumptions \ref{Ass-meas-exists:as}, \ref{uniform-link:as} and \ref{coercivity:as} follow from Remark \ref{compact:rm}.
\end{proof}


\subsection{An Additional Assumption for Asymptotic Normality.}

Again, let $j\in\{1,\ldots,m-1\}$.
\begin{As}\label{as:clt-uniform-derivatives}
  Assume that $T_{m-1,m-j}$ carries a smooth manifold structure near the unique BN population mean ${f'}^{j-1} = ( {p'}^{m-1},\ldots, {p'}^{j-1})$ such that there is an open set $W\subset T_{m-1,m-j}$, ${f'}^{j-1} \in W$ and a local chart
  \begin{align*}
    \psi: W\to \RR^{\dim(U)},~f^{j-1}=({p}^{m-1},\ldots, {p}^{j-1})\mapsto \eta\,.
  \end{align*}
%
   Further, 
   assume that for every $l=j,\ldots,m$ the mapping
  \begin{align*}
    \eta \mapsto f^{l-1}\mapsto \rho_{p^{l}}(\pi_{f^l}\circ X, p^{l-1})^2:=\tau^l(\eta,X)
  \end{align*}
  has first and second derivatives, such that for all $l=j,\ldots,m$,
  \begin{align*}
    \cov\big[\grad_\eta \tau^l(\eta',X)\big]\mbox{, and }\mathbb E\big[\Hess_\eta \tau^l(\eta',X)\big]\,
  \end{align*}
  exist and are in expectation continuous near $\eta'$, i.e. for $\delta\to 0$ we have 
  \begin{align*}
    \EE\left[\mathop{\sup}_{\|\eta-\eta'\|<\delta} \left\|\grad_\eta \tau^l(\eta,X) - \grad_\eta \tau^l(\eta',X)\right\|\right]\:\to\:0\,,
    \\
    \EE\left[\mathop{\sup}_{\|\eta-\eta'\|<\delta}\left\|\Hess_\eta \tau^l(\eta,X) - \Hess_\eta \tau^l(\eta',X)\right\|\right]\:\to\:0\,.
  \end{align*}
  Finally, assume that the vectors $\EE\big[\grad_\eta \tau^{j+1}(\eta',X)\big],\ldots, \EE\big[\grad_\eta \tau^{m}(\eta',X)\big]$ are linearly independent.
\end{As}

\begin{Rm}
  For PNS and PNGS a global, manifold structure has been derived in Section \ref{sec:PNS-Framework} with projections (\ref{eq:blow-down}) (see also Proposition \ref{PNS-BNDS:prop}) and distances (\ref{rho-p-PNS:eq}) smooth away from singularity sets. For IMo1GPCs on Kendall's shape spaces, this has been provided in Section \ref{scn:Kendalls-shape-spaces}, cf. also \cite{H_ziez_geod_10}.

  In general, however, it is unclear under which circumstances (if the second derivatives are continuous in both arguments where $X$ is supported in a compact set, then convergence to zero holds not only in expectation but also a.s.) the three assumptions above, uniqueness, existence of first and second moments of second and first derivatives and their continuity in expectation are valid. Even for the much simpler case of intrinsic means on manifolds this is only very partially known, cf. the discussion in \cite{HH13}. It seems that only for the most simple non-Euclidean case of intrinsic means on circles the full picture is available (\cite{HH15}). Recently, rather generic conditions for densities have been derived by \cite{BL16}, ensuring $\sqrt{n}$-Gaussian asymptotic normality.
  
  The condition on linear independence is rather natural for realistic scenarios where each constraining condition adds a new constraint, not covered by the previous, as introduced after Corollary \ref{cor:lambda-lem2}. For example, if charts factor, then with decreasing $l$, every constraning condition results in conditions on new coordinates. 
\end{Rm}

\section{The Main Results}\label{main-results:scn}

\subsection{Asymptotic Theorems}\label{asymptotics:scn}

\begin{Th}\label{SC:thm} 
  Let $k\in \{0,\ldots,m-1\}$ and consider random data $X_1,\ldots,X_n \stackrel{iid}{\sim }X$ on a data space $Q$ admitting BN descriptor families from $P_{m}$ to $P_k$, unique BN population means $\{p^{m},\ldots,p^k\}$ and BN sample means $\{E_n^{f_n^{m}},\ldots,E_n^{f_n^k}\}$ due to a measurable selection $p^{j}_n \in E^{f_n^{j}}_n$ giving rise to BNFDs $f^j_n=\{p^l_n\}_{l=j}^{m}$, $j=k,\ldots,m$. If Assumptions \ref{Ass-meas-exists:as} -- \ref{coercivity:as} are valid for all $j=k,\ldots,m-1$, and every $\overline{\cup_{n=1}^\infty E^{f^j_n}_n}$ is a.s. $d_j$-Heine Borel ($j=k,\ldots,m$) then $\{E_n^{f_n^{m}},\ldots,E_n^{f_n^k}\}$ converges a.s. to $\{p^{m},\ldots,p^k\}$ in the sense that $\exists \Omega' \subset \Omega$ measurable with $\Prb(\Omega')=1$ such that for all $j=k,\ldots,m$, $\epsilon >0$ and $\omega\in \Omega'$, $\exists N =N(\epsilon,\omega)$ with
  \begin{align}\label{BP-SC:eq}
    \bigcup_{r=n}^\infty E^{f_r^j}_r\subset \{p\in P_j: d_j(p^j,p)\leq \epsilon\}\quad \forall n \geq N,~ \omega\in \Omega'\,.
  \end{align}	
\end{Th}

\begin{proof}
  We proceed by backward induction on $j$. The case $j=m$ is trivial and the case $j=m-1$ has been covered by Theorems A.3 and A.4 from \cite{H_ziez_geod_10}.
  
  Now suppose that (\ref{BP-SC:eq}) have been established for $j+1 \in \{k+1,\ldots,m\}$.
  Set $P=P_{j+1}, p = p^{j+1}, f=\{p^{m},\ldots,p^{j+1}\}, p_n = p_n^{j+1}, f_n=\{p_n^{m},\ldots,p_n^{j+1}\}$ and for an arbitrary BNFD $f'$ ending at $p'\in P_{j+1}$
  \begin{align*}
    F_f(s) &= \mathbb E\big[\rho_{p}(\pi_f\circ X,s)^2\big],~s\in S_p& F_{n,f'}(s) &= \frac{1}{n}\sum_{i=1}^n \rho_{p'}(\pi_{f'}\circ X_i,s)^2,~s\in S_{p'}\\
    \lw_f &= \inf_{s\in S_p} F_f(s),& \lw_{n,f'} &= \inf_{s\in S_{p'}} F_{n,f'}(s)
  \end{align*}
  Then, $F_f(s) <\infty$ for all $s\in S_p$, by hypothesis, and with $s^*=p^{j}$,
  \begin{align*}
    \{s^*\} = \argmin_{s\in S_p}F_f(s),&& E_n^{f_n} = \argmin_{s\in S_{p_n}}F_{n,f_n}(s)\,. 
  \end{align*}
  To complete the proof we first show in the Appendix
  \begin{align}\label{Ziezold-sc:eq}
    \bigcap_{n=1}^\infty \overline{\bigcup_{r=n}^\infty E^{f_r}_r} \subset \{s^*\}\mbox{ a.s. }\,.
    \end{align}
  This is Ziezold's version of strong consistency (cf. \cite{Z77}). Further, we show that this implies the Bhattacharya-Patrangenaru version (cf. \cite{BP03}) of strong consistency which takes here the form (\ref{BP-SC:eq}).
\end{proof}

\begin{Rm}
   Careful inspection of the proof yields that we have only used that the ``distances'' $d_j$ vanish on the diagonal $d_j(p,p)=0$ for all $p\in P_j$;
   they need not be definite, i.e. it is not necessary that $d_j(p,p')=0\Rightarrow p=p'$. 

   Moreover, 
  note that the $d_j$-Heine Borel property holds trivially in case of unique sample descriptors.
\end{Rm} 

\begin{Cor}\label{cor:lambda-lem2}
  Suppose that (\ref{BP-SC:eq}) holds together with Assumption 
  \ref{as:clt-uniform-derivatives}. Then we have for $l=1,\ldots,m$ the following convergence in probability
  \begin{align*}
    \frac{1}{n}\sum_{k=1}^n \grad_\eta\tau^l(\eta_n,X_k) &\:\stackrel{\Prb}{\to}\: \EE\left[ \grad_\eta\tau^l(\eta',X)\right]\,,&
    \frac{1}{n}\sum_{k=1}^n \Hess_\eta\tau^l(\eta_n,X_k) &\:\stackrel{\Prb}{\to}\: \EE\left[ \Hess_\eta\tau^l(\eta',X)\right]\,.
  \end{align*}
\end{Cor}

\begin{proof}
  Let $\epsilon >0$. Then by Assumption \ref{as:clt-uniform-derivatives}, Chebyshev's inequality and (\ref{BP-SC:eq}), there is a sequence $\delta_n \to 0$ such that $\|\eta_n - \eta'\|<\delta_n\as$ and
  \begin{align*}
    \MoveEqLeft
    \Prb\left\{\left\| \frac{1}{n}\sum_{k=1}^n\grad_\eta\tau^l(\eta_n,X_k) - \EE\big[\grad_\eta\tau^l(\eta',X)\big]\right\| \geq \epsilon \right\}
    \\
    &\leq \Prb\left\{\sup_{\|\eta-\eta'\|<\delta_n}\left\| \frac{1}{n}\sum_{k=1}^n\grad_\eta\tau^l(\eta,X_k) - \EE\big[\grad_\eta\tau^l(\eta',X)\big]\right\| \geq \epsilon \right\}\\
    &\leq \frac{1}{\epsilon}\,\EE\left[\sup_{\|\eta-\eta'\|<\delta_n}\left\| \frac{1}{n}\sum_{k=1}^n\grad_\eta\tau^l(\eta,X_k) - \EE\big[\grad_\eta\tau^l(\eta',X)\big]\right\| \right]~\to~0
  \end{align*}
  as $n\to \infty$. yielding the first assertion. The second assertion follows similarly.
\end{proof}

We now introduce notation we use for the central limit theorem. Let $j \in \{1,\ldots,m\}$. By construction, every measurable selection $f_n^{j-1}$ of BN sample means minimizes
\begin{align*}
  \frac{1}{n} \sum_{k=1}^n \rho_{p^j_n}(\pi_{f^j_n}\circ X_k,p^{j-1}_n)^2
\end{align*}
under the constraints that it minimizes each of
\begin{align*}
  \frac{1}{n} \sum_{k=1}^n \rho_{Q}(X_k,p^{m-1}_n)^2, \quad \ldots \quad , \frac{1}{n} \sum_{k=1}^n \rho_{p^{j+1}_n}(\pi_{f^{j+1}_n}\circ X_k,p^{j}_n)^2\,.
\end{align*}
Similarly, the BN population mean ${f'}^{j-1}$ minimizes
\begin{align*}
  \EE\left[ \rho_{p^j}(\pi_{f^j}\circ X,p^{j-1})^2\right]
\end{align*}
under the constraints that it minimizes each of
\begin{align*}
  \EE\left[ \rho_{Q}(X,p^{m-1})^2\right], \quad\ldots\quad , \EE\left[ \rho_{p^{j+1}}(\pi_{f^{j+1}}\circ X,p^{j})^2\right]\,.
\end{align*}
In consequence, due to differentiability guaranteed by Assumption \ref{as:clt-uniform-derivatives}, with the notation of $\tau^j$ there, suitable random Lagrange multipliers $\lambda^{j+1}_n,\ldots,\lambda^m_n \in \RR$ and deterministic Lagrange multipliers $\lambda^{j+1},\ldots,\lambda^m \in \RR$ exist such that for $\eta_n = \psi^{-1}(f_n^{j-1})$ and $\eta' = \psi^{-1}({f'}^{j-1})$ the following hold
\begin{alignat}{7}
  &&\grad_{\eta}G_n(\eta_n) &=0&\mbox{ with }&&G_n(\eta) &:= \frac{1}{n} \sum_{k=1}^n \tau^j(\eta,X_k) &&+ \sum_{l=j+1}^m \lambda_n^l\,\frac{1}{n} \sum_{k=1}^n \tau^l(\eta,X_k)  \label{eq:sample-Lagrange}\\
  &&\grad_{\eta}G(\eta') &=0 &\mbox{ with }&&G(\eta) &:= \EE\big[\tau^j(\eta,X)\big] &&+ \sum_{l=j+1}^m \lambda^l \,\EE\big[\tau^l(\eta,X)\big]   \label{eq:population-Lagrange}
\end{alignat}

\begin{Cor}\label{lem:lambda-lem1}
  Suppose that (\ref{BP-SC:eq}) holds together with Assumption 
  \ref{as:clt-uniform-derivatives}. Then the random Lagrange multipliers $\lambda^{j+1}_n,\ldots,\lambda^m_n$ in (\ref{eq:sample-Lagrange}) and $\lambda^{j+1},\ldots,\lambda^m$ in (\ref{eq:population-Lagrange}) satisfy
  \begin{align*}
    \lambda^{l}_n \stackrel{\Prb}{\to} \lambda^{l} \mbox{ for }l=j+1,\ldots,m\,.
  \end{align*}
\end{Cor}

\begin{proof}
  By hypothesis, the vector
  \begin{align*}
    a_n:= \frac{1}{n} \sum_{k=1}^n \grad_{\eta}\,\tau^j(\eta_n,X_k)
  \end{align*}
  is a linear combination of the vectors
  \begin{align*}
    b_n^l :=\frac{1}{n} \sum_{k=1}^n \grad_\eta \,\tau^l(\eta_n,X_k),~~l=j+1,\ldots,m
  \end{align*}
  conveyed by $\lambda^{j+1}_n,\ldots,\lambda^m_n$. Similarly, $a:=\EE[ \grad_{\eta}\,\tau^j(\eta',X)]$ is a linear combination of the vectors $b^l:=\EE[ \grad_{\eta}\,\tau^l(\eta',X)] $, $l=j+1,\ldots,m$ conveyed by $\lambda^{j+1},\ldots,\lambda^m$. Set $b = (b^{j+1},\ldots,b^m)$, $b_n = (b_n^{j+1},\ldots,b_n^m)$ and $\lambda = (\lambda^{j+1},\ldots,\lambda^m)^T$, $\lambda_n = (\lambda_n^{j+1},\ldots,\lambda_n^m)^T$.
  
  By Assumption \ref{as:clt-uniform-derivatives} we have $\rank(b) = m-j$. We set $\Omega_n = \{\rank(b_n) = m-j\}$. Since the determinant is continuous, by Corollary \ref{cor:lambda-lem2}, 
  \begin{align*}
    \Prb(\Omega_n) \to 1\mbox{ as }n\to \infty\,.
  \end{align*}
  Now consider the function $g(b,a) = b^{+} a$ where $b^+$ denotes the Moore-Penrose pseudoinverse of $b$. Then $\lambda = b^{+}a$ and $\lambda_n = b_n^+ a_n$.
  In consequence, for arbitrary $\epsilon >0$ and $n\to \infty$ we have that
  \begin{align*}
    \Prb\{\|\lambda_n - \lambda\| > \epsilon\}\leq \Prb\Big\{\|g(b_n,a_n) -g(b,a)\| > \epsilon, \omega\in \Omega_n\Big\} + \Prb(\Omega \setminus \Omega_n)\to 0
  \end{align*}
  because, for any $\delta >0$ the first term is smaller than
  \begin{align*}
    \Prb\Big\{\|g(b_n,a_n) -g(b,a)\| > \epsilon, \omega\in \Omega_n, \|(b_n,a_n) - (b,a)\| \leq \delta\Big\} + \Prb\Big\{\|(b_n,a_n) - (b,a)\| >\delta\Big\}\,.
  \end{align*}
  Here, due to continuity established by \cite{Stewart1969}, the first term vanishes for $\delta$ sufficiently small, and, due to Corollary \ref{cor:lambda-lem2}, for any fixed $\delta >0$, the second term tends to zero. This yields the assertion.
%
\end{proof}

\begin{Th}\label{CLT:thm}
  Let $j\in \{1,\ldots,m-1\}$ and consider random data $X_1,\ldots,X_n \stackrel{iid}{\sim }X$ on a data space $Q$ admitting BNFDs from $P_{m-1}$ to $P_{j-1}$, a unique BN population mean ${f'}^{j-1} = \{{p'}^{m-1},\ldots,{p'}^{j-1}\}$ and BN sample means $\{E_n^{f_n^{m-1}},\ldots,E_n^{f_n^{j-1}}\}$ due to a measurable selection $p^{l}_n \in E^{f_n^{l}}_n$, $f^{j-1}_n = \{p_n^{m-1},\ldots,p_n^{j-1}\}$, $l=j-1,\ldots,m-1$.
  
  \begin{itemize}
    \item[(i)]
      Assuming that Assumption 
      \ref{as:clt-uniform-derivatives} hold as well as (\ref{BP-SC:eq}) for all $j\in \{j-1,\ldots,m-1\}$, we have that
      \begin{align*} 
        \sqrt{n}H_\psi \big(\psi^{-1}(f^{j-1}_n) - \psi^{-1}({f'}^{j-1})\big)\to{\cal N}(0,B_{\psi})
      \end{align*}
      with a chart $\psi$ as specified in Assumption  \ref{as:clt-uniform-derivatives} 
      as well as
      \begin{align*}
        H_{\psi} &\:=\: \EE\left[\Hess_{\eta}\tau^j(\eta',X)+ \sum_{l=j+1}^m \lambda^l\, \Hess_{\eta}\tau^l(\eta',X) \right] \mbox{ and }\\
        B_{\psi} &\:=\: \cov\left[ \grad_{\eta}\tau^j(\eta',X) + \sum_{l=j+1}^m \lambda^l\, \grad_{\eta}\tau^l(\eta',X)\right] \,,
      \end{align*}
      with the notation from Assumption \ref{as:clt-uniform-derivatives} where $\lambda_{j+1},\ldots\lambda_m\in \RR$ are suitable such that
      \begin{align*}
        \grad_{\eta}\, \EE\big[\tau^j(\eta,X)\big] + \sum_{l=j+1}^m \lambda^l\,\grad_\eta \,\EE\big[\tau^l(\eta,X)\big]\,
      \end{align*}
      vanishes at $\eta=\eta'$.
    \item[(ii)]
      If additionally $H_{\psi}>0$, then 
      $f^{j-1}_n$ satisfies a Gaussian $\sqrt{n}$-CLT 
      \begin{align*} 
        \sqrt{n} \big(\psi^{-1}(f^{j-1}_n) - \psi^{-1}({f'}^{j-1})\big)\to{\cal N}(0, \Sigma_{\psi}),\quad
        \Sigma_{\psi} = H_\psi^{-1} B_\psi H_\psi^{-1} \,.
      \end{align*}
      \item[(iii)] If additionally the chart $\psi$ factors as in Definition \ref{def:factoring}, then also $p^{j-1}_n$ satisfies a Gaussian $\sqrt{n}$-CLT 
      \begin{align*}
        \sqrt{n} \big(\phi^{-1}(p^{j-1}_n) - \phi^{-1}({p'}^{j-1})\big)\to{\cal N}(0, \Sigma_{\phi}),\quad \Sigma_{\phi} = \big({\Sigma_{\psi}}_{ik}\big)_{i,k=1}^{\dim(P_{j-1})} \,
      \end{align*}
      with the notation of Definition \ref{def:factoring}. 
  \end{itemize}

\end{Th}

\begin{proof}
  By Taylor expansion we have for $G_n$ defined in (\ref{eq:sample-Lagrange}), 
  \begin{alignat}{3}
    0&= \sqrt{n}\,\grad_{\eta}G_n(\eta_n) \vphantom{\smash[t]{\Big)}}\nonumber\\
    &= \sqrt{n}\,\grad_{\eta}G_n(\eta') ~&&+~ \Hess_{\eta}G_n(\eta') \sqrt{n}\, (\eta_n - \eta') \label{eq:clt-proof-1a}\\
    &&&+~ \Big(\Hess_{\eta}G_n(\widetilde{\eta}_n) - \Hess_{\eta}G_n(\eta') \Big) \sqrt{n}\, (\eta_n - \eta') \nonumber
  \end{alignat}
  with some random $\widetilde{\eta}_n$ between $\eta_n$ and $\eta'$. In consequence of Corollary \ref{lem:lambda-lem1} we have with the usual CLT for the first term in (\ref{eq:clt-proof-1a}) that 
  \begin{align*}
    \sqrt{n}\, \grad_{\eta}G_n(\eta') &\:=\: \frac{1}{\sqrt{n}}\sum_{k=1}^n \left( \grad_{\eta}\tau^j(\eta',X_k) + \sum_{l=j+1}^m \big(\lambda^l+o_p(1)\big)\, \grad_{\eta}\tau^l(\eta',X_k) \right) \to~ \cG
  \end{align*}
  with a zero-mean Gaussian vector $\cG$ of covariance
  \begin{align*}
    \cov\left[ \grad_{\eta}\tau^j(\eta',X) + \sum_{l=j+1}^m \lambda^l\, \grad_{\eta}\tau^l(\eta',X)\right]\,.
  \end{align*}
  Similarly, we have for the first factor in the second term in (\ref{eq:clt-proof-1a}),
  \begin{align*}
    \Hess_{\eta}G_n(\eta') & \:=\: \frac{1}{n} \sum_{k=1}^n \left( \Hess_{\eta}\tau^j(\eta',X_k) + \sum_{l=j+1}^m \lambda^l_n\, \Hess_{\eta}\tau^l(\eta',X_k) \right)\\ 
    &\:\stackrel{\Prb}{\to}\: \EE\left[\Hess_{\eta}\tau^j(\eta',X)+ \sum_{l=j+1}^m \lambda^l\, \Hess_{\eta}\tau^l(\eta',X) \right]\,.
  \end{align*}
  Finally, for the first factor in the last the term in (\ref{eq:clt-proof-1a}), invoking also Corollary \ref{cor:lambda-lem2}, we obtain that
  \begin{align*}
    \left\| \Hess_{\eta}G_n(\widetilde{\eta}_n) - \Hess_{\eta}G_n(\eta') \right\| \stackrel{\Prb}{\to} 0\,.
  \end{align*}
  This yields Assertion (i). If $H_\psi$ is invertible, as asserted in (ii), joint normality follows at once for $\sqrt{n}(\eta_n -\eta')$. 
  
  (iii): In case of factoring charts we can rewrite
  \begin{align*}
    \rho_{p^{j}}(\pi_{f^j}\circ X, p^{j-1})^2
    &= \rho_{\pi^{P_{j}} \circ \psi_2^{-1}(\xi)} \Big(\pi_{\psi_2^{-1}(\xi)}\circ X, \psi_1^{-1}(\theta)\Big)^2=:\tau^j(\theta,\xi,X)
  \end{align*}
  where $\psi(f^{j-1}) = \psi(f^j,p^{j-1}) = \big(\psi_1(f^{j-1}),\psi_2(f^{j})\big) = \eta$ with $\psi_1(f^{j-1}) = \theta = \phi(p^{j-1})$ and $\psi_2(f^{j}) = \xi$, and $\pi^{P_j} : T_{m-1,m-j+1}\to P_j$ is defined by $(p^{m-1},\ldots,p^{j}) \mapsto p^j$. With $(\theta_n,\xi_n) = \eta_n$ in Assertion (ii) we obtain thus
  \begin{align*}
    \sqrt{n} H_\psi \left(\begin{pmatrix} \theta_n\\ \xi_n\end{pmatrix}-\eta'\right) \to \cN(0,B_\psi)\,.
  \end{align*}
  Since under projection to the first coordinates $\theta = \phi(p^{j-1})$, asymptotic normality is preserved, Assertion (iii) follows at once. 
\end{proof}

\subsection{A Nested Two-Sample Bootstrap Test}
Suppose that we have two independent i.i.d. samples $X_1,\ldots,X_n\sim X \in Q$, $Y_1,\ldots,Y_m\sim Y \in Q$ in a data space $Q$ admitting BNFDs and we want to test 
\begin{align*}
  H_0: X\sim Y \quad\mbox{versus}\quad H_1: X\not \sim Y
\end{align*} 
using descriptors in $p \in P$. Here, $p \in P$ stands either for a single $p_j\in P_j$ for which we have established factoring charts, or for a suitable sequence $f \in T_{j,k}$. We assume that the first sample gives rise to $\hat p^X_n\in P$, the second to $\hat p^Y_m\in P$, and that these are unique. Under the corresponding assumptions of Theorem \ref{CLT:thm}, 
define a statistic
\begin{align*}
  T^2(A) = \big(\phi(\hat p_n^X )- \phi(\hat p_m^Y )\big)^T A \big(\phi(\hat p_n^X )- \phi(\hat p_m^Y )\big)\,.
\end{align*}
Under $H_0$, up to a suitable factor, this is Hotelling $T^2$ distributed if $A^{-1}$ is the corresponding empirical covariance matrix. Therefore, for $A^{-1}$ we use the empirical covariance matrix from bootstrap samples.

With this fixed $A$, we simulate that statistic under $H_0$ by again bootstrapping $B$ times. Namely from $X_1,\ldots,X_n,Y_1\ldots,Y_m$ we sample $Z_{1,b},\ldots,Z_{n+m,b}$ and compute the corresponding
$ {T^*}^2(A)_b$ ($b = 1,\ldots,B$) from $X^*_{i,b} = Z_{i,b}, Y^*_{j,b} = Z_{n+j,b}$ ($i=1,\ldots,n$, $j=1,\ldots,m$). From these, for a given level $\alpha \in (0,1)$ we compute the empirical quantile $c^*_{1-\alpha}$ such that
\begin{align*}
  \Prb\big\{ {T^*}^2(A) \leq c^*_{1-\alpha}|X_1,\ldots,X_n,Y_1,\ldots,Y_m\big\} = 1-\alpha\,.
\end{align*}
Arguing as in \cite[Corollary 2.3 and Remark 2.6]{BP05} which extends at once to our setup, we assume that the corresponding population covariance matrix $\Sigma_\psi$ or $\Sigma_\phi$, respectively, from Theorem \ref{CLT:thm} is invertible. We have then under $H_0$ that $c^*_{1-\alpha}$ gives an asymptotic coverage of $1-\alpha$ for $T^2(A)$, i.~e. $ \Prb\{T^2(A) \leq c^*_{1-\alpha}\}\to 1-\alpha$ as $n,m \to \infty$ if $n/m \to c$ with a fixed $c \in (0,\infty)$.

\section{Applications}\label{application:scn}

\subsection{Simulations}

To illustrate our CLT for principle nested spheres (PNS) and principle nested great spheres (PNGS), we simulate three data sets, each from two paired random variables $X$ and $Y$, displayed in Figure \ref{simulated_data}.
\begin{enumerate}[I)]
  \item Data on an $S^3$ concentrate on the same proper small $S^2$ and there on segments of orthogonal great circles such that their nested means are antipodal.
  \item Data on an $S^3$ concentrate on the same proper small $S^2$ and there on segments of orthogonal great circles such that their nested means coincide.
  \item Data on an $S^2$ concentrate on segments of different small circles, have different nested means under PNS, but, under PNGS,  
  coinciding principal geodesics and nested means.
\end{enumerate}

\begin{figure}[ht!]
  \centering
  \subcaptionbox{Data set I}[0.3\textwidth]{\includegraphics[width=0.3\textwidth, clip=true, trim=3.5cm 1cm 3cm 1cm]{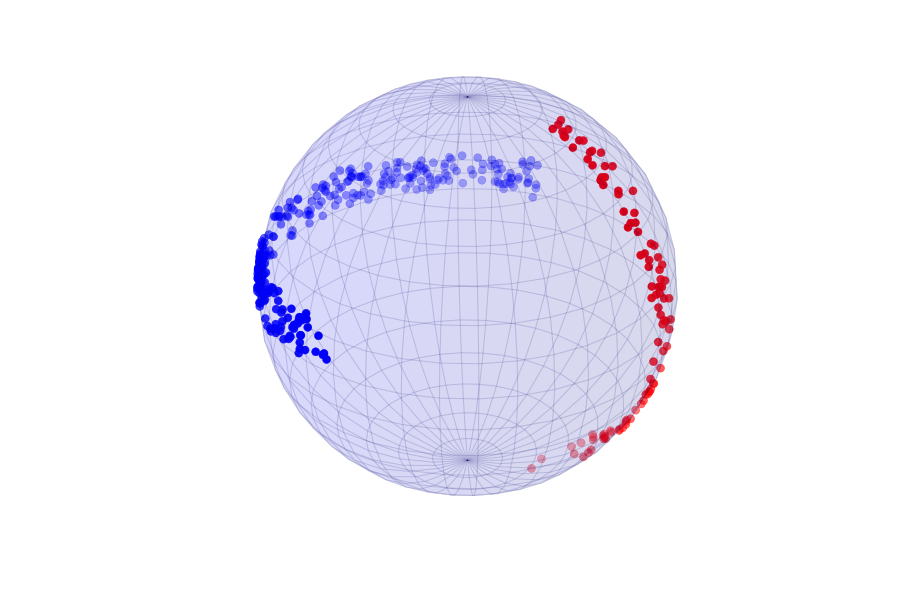}}
  \subcaptionbox{Data set II}[0.3\textwidth]{\includegraphics[width=0.3\textwidth, clip=true, trim=3.5cm 1cm 3cm 1cm]{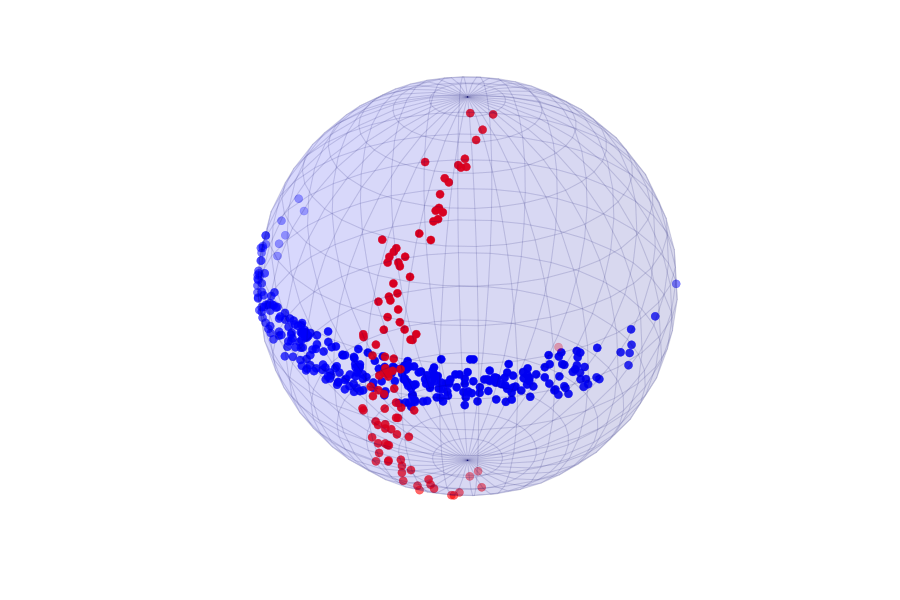}}
  \subcaptionbox{Data set III}[0.3\textwidth]{\includegraphics[width=0.3\textwidth, clip=true, trim=3.5cm 1cm 3cm 1cm]{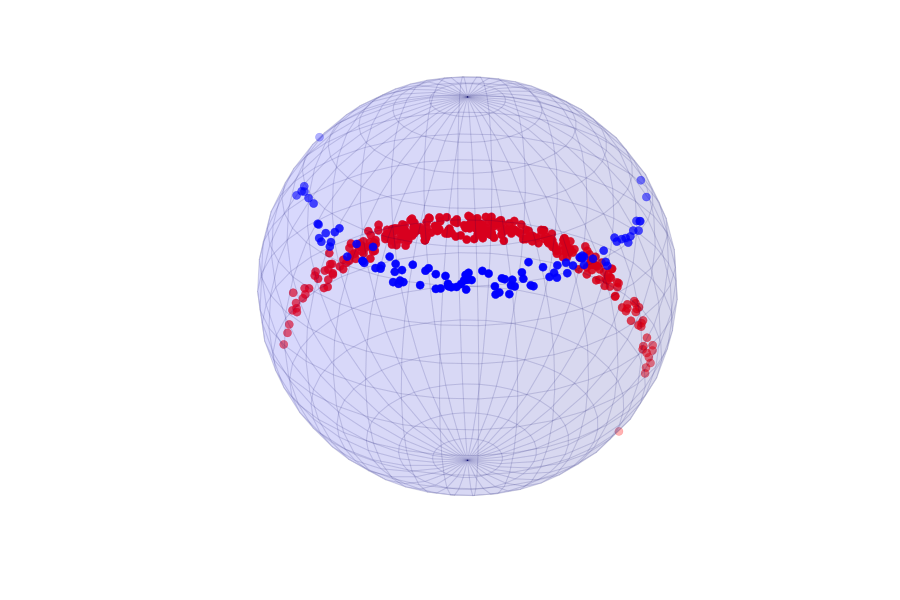}}
  \caption{\it Simulated datasets I (left) and II (middle) on $S^3$ concentrate on a common proper small $S^2$, their projections to estimated small two-spheres is depicted. The simulated dataset III (right) is on $S^2$.}
  \label{simulated_data}
\end{figure}

We apply PNS and PNGS to the simulated data and perform the two-sample test for identical respective nested submanifolds (means, small and great circles) and for identical small and great two-spheres. The resulting p-values are displayed in Table \ref{table}. These values are in agreement with the intuition guiding the design of the data.

\begin{table}[!ht]
  \caption{\textit{Displaying p-values for PNS and PNGS from the two-sample test on identical nested mean (column 0d), on identical nested small and great circle, respectively, (column 1d) and on identical small and great two-sphere (column 2d). Using $B=1000$ bootstrap samples, the penultimate p-value is $10^{-3}$.}}
  \begin{center}
    \begin{tabular}{|l|l|r|r|r|}
    \hline
    \rule{0pt}{2.6ex} Data Set & Method & $0$d       & $1$d       & $2$d       \\ \hline
    \rule{0pt}{2.6ex} I        & PNS    & $<10^{-3}$ & $<10^{-3}$ & $0.50$     \\
    \rule{0pt}{2.4ex}          & PNGS   & $<10^{-3}$ & $<10^{-3}$ & $<10^{-3}$ \\ \hline
    \rule{0pt}{2.6ex} II       & PNS    & $0.50$     & $<10^{-3}$ & $0.20$     \\
    \rule{0pt}{2.4ex}          & PNGS   & $0.47$     & $<10^{-3}$ & $<10^{-3}$ \\ \hline
    \rule{0pt}{2.6ex} III      & PNS    & $<10^{-3}$ & $<10^{-3}$ &            \\
    \rule{0pt}{2.4ex}          & PNGS   & $0.07$     & $0.15$     &            \\ \hline
    \end{tabular}
  \end{center}
  \label{table}
\end{table}

\subsection{Early Human Mesenchymal Stem Cell Differentiation}

Understanding differentiation of adult human stem cells with the perspective of clinical use (see e.g. \cite{Pittengeretal1999} who emphasize their potential for cartilage and bone reconstruction) is an ongoing fundamental challenge in current medical research, still with many open questions (e.g. \cite{Biancoetal2013}). To investigate mechanically guided differentiation, \emph{human mesenchymal stem cells} (hMSCs, pluripotent adult stem cells taken from the bone marrow) are placed on gels of varying elasticity, quantified by the Young's modulus, to mimic different environments in the human body, e.g. \cite{DischerJanmeyWang2005}. It is well known that within the first day the surrounding elasticity measured in kilopascal (kPa) induces differentiation through biomechanical cues, cf. \cite{EnglerSenSweeneyDisher2006, ZemelRehfeldBrownDischerSafran2010NatPhys}, where the changes manifest in orientation and ordering of the \emph{actin-myosin filament skeleton}. In particular, in order to direct future, more focused research, it is of high interest to more precisely identify time intervals in which such changes of ordering occur and to separate changes due to differentiation from changes due to other causes.
\begin{figure}[h!]
  \centering
  \subcaptionbox{Original}[0.48\textwidth]{\includegraphics[angle = 0, width=0.48\textwidth]{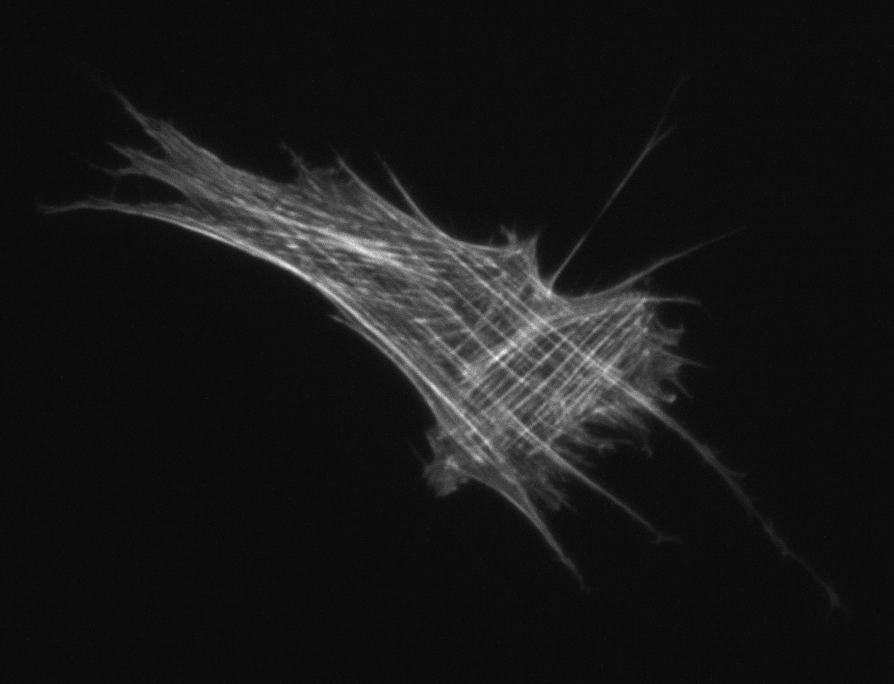}}
  \hspace*{0.02\textwidth}
  \subcaptionbox{All detected filaments}[0.48\textwidth]{\includegraphics[angle = 0, width=0.48\textwidth]{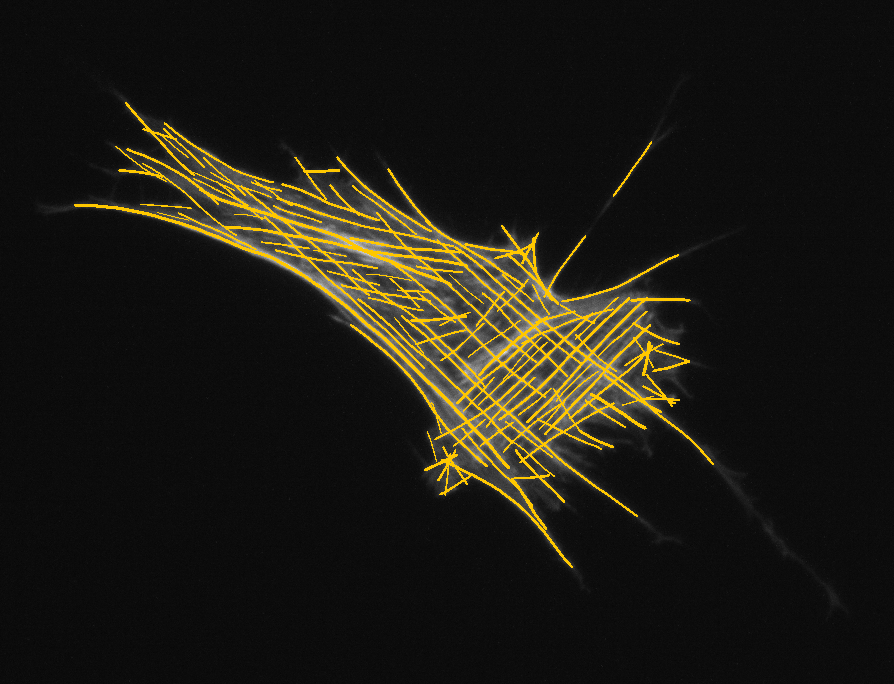}}\\
  \subcaptionbox{Main field}[0.48\textwidth]{\includegraphics[angle = 0, width=0.48\textwidth]{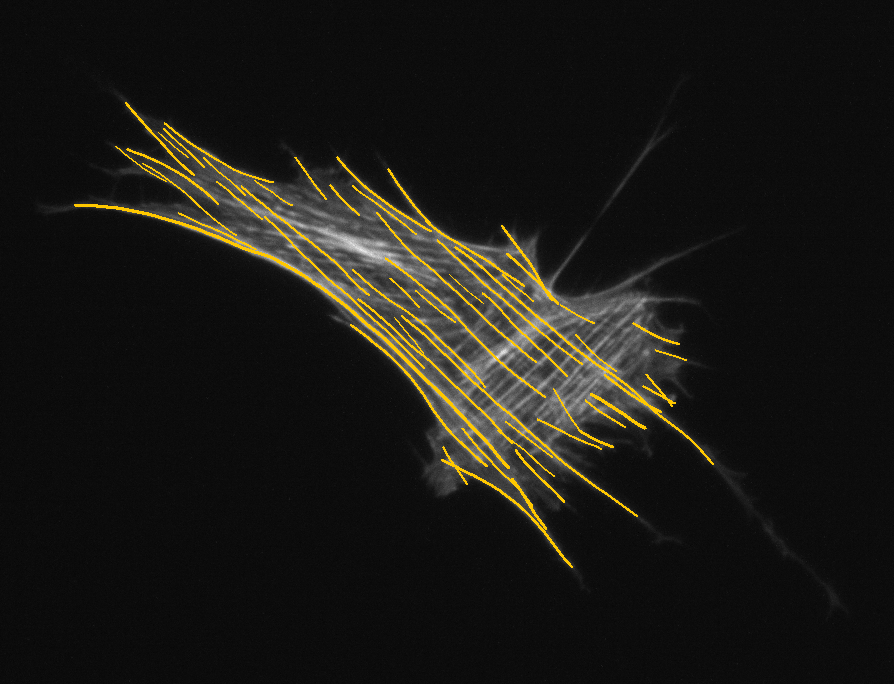}}
  \hspace*{0.02\textwidth}
  \subcaptionbox{Smaller fields and other filaments}[0.48\textwidth]{\includegraphics[angle = 0, width=0.48\textwidth]{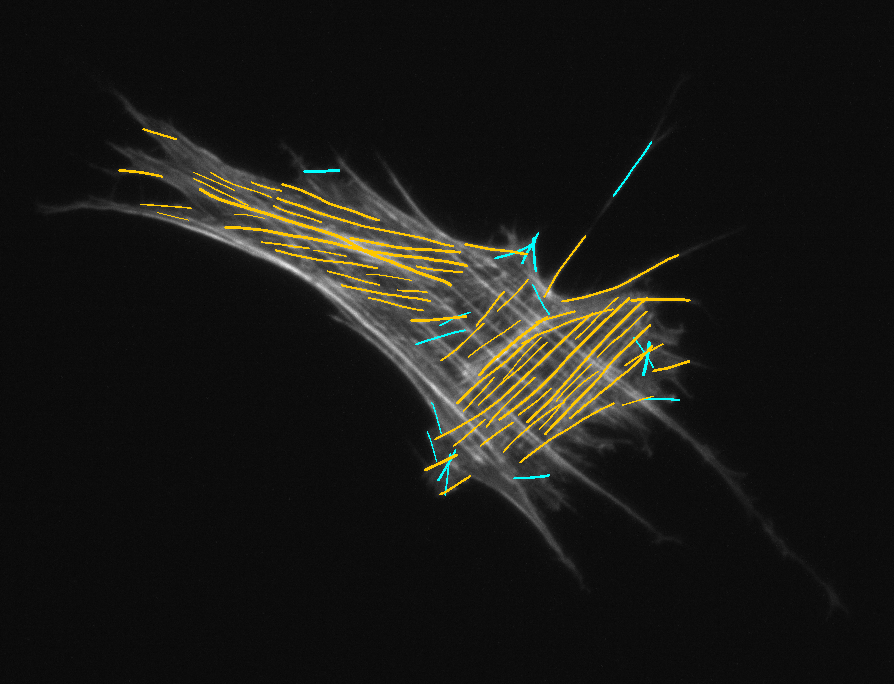}}
  \caption{\it (a): Fluorescence image of an immuno-stained human mesenchymal stem cell after cultivation for 16 hours at Young's modulus 10~kPa. (b): Automatically extracted filament structure using the Filament Sensor from \cite{EltznerWollnikGottschlichHuckemannRehfeldt2015}. (c): Filaments of the largest orientation field. (d): Filaments of smaller orientation fields (yellow) and filaments not belonging to any orientation field (cyan).}
  \label{fig:filament-sensor}
\end{figure}

\textbf{Experimental setup.}
We compare hMSC skeletons that have been cultured at the Third Institute of Physics of the University of G\"ottingen on gels with Young's moduli of 1~kPa mimicking neural tissue, 10~kPa mimicking muscle tissue, and 30~kPa mimicking bone tissue. The cells have been fixed after multiples of 4 hours on the respective gel and have then been immuno-stained for NMM IIa, the motor proteins making up small filaments that are responsible for cytoskeletal tension and imaged (as described in \cite{ZemelRehfeldBrownDischerSafran2010NatPhys}). Table \ref{table_samples} shows their sample sizes and the data will be published and made available after completion of current research, cf. \cite{WollnikRehfheltd2016}. Because earlier research (\cite{WiZer2016}) suggests that during the first 24 hours, 10~kPa and 30~kPa hMSCs develop rather similarly and quite differently from 1~kPa hMSCs, for this investigation, we pool the former.

\begin{table}[!ht]
  \caption{\textit{Sample sizes of hMSC skeleton images over varying Young's moduli and cultivation time.}}
  \begin{center}
    \begin{tabular}{|l|c|c|}
    \hline
    \rule{0pt}{2.6ex} Time & 1~kPa  & 10~kPa and 30~kPa \\ \hline
    \rule{0pt}{2.6ex}  4h  & 159    & 321 \\
    \rule{0pt}{2.6ex}  8h  & 163    & 317 \\
    \rule{0pt}{2.6ex} 12h  & 176    & 344 \\
    \rule{0pt}{2.6ex} 16h  & 135    & 274 \\
    \rule{0pt}{2.6ex} 20h  & 138    & 253 \\
    \rule{0pt}{2.6ex} 24h  & 166    & 304 \\ \hline
    \end{tabular}
  \end{center}
  \label{table_samples}
\end{table}

The actin-myosin filament structure has been automatically retrieved from the fluorescence images using the Filament Sensor from \cite{EltznerWollnikGottschlichHuckemannRehfeldt2015}. Since neighboring filaments share the same orientation, the 3D structure of the cellular skeleton can be retrieved by separating the filament structure into different orientation fields, cf. Figure \ref{fig:filament-sensor}.

\textbf{Orientation fields} for filament structures are determined via a relaxation labeling procedure, see \cite{Rosenfeld1976}. 
The source code of our implementation is available as supplementary material. 
A detailed description is deferred to a future publication. The algorithm results in a set of contiguous areas with slowly varying local orientation, and, corresponding to each of these areas, a set of filaments which closely follow the local orientation. Also, these data will be published and made available after completion of current research, cf. \cite{WollnikRehfheltd2016}.

\textbf{Data analysis.}
For each single hMSC image, let $M$ be the number of pixels of all detected filaments, $m_1$ the number of all filament pixels of filaments of the largest orientation field and $m_2$ the number of all filament pixels of filaments of all smaller orientation fields. $M-m_1-m_2$ is then the number of pixels in all ``rogue'' filaments which are not associated to any field, because they are too inconsistent with neighboring filaments. Define $x=(x_1,x_2,x_3):= (\sqrt{m_1/M}, \sqrt{m_2/M}, \sqrt{1 - (m_1 + m_2)/M})^T \in Q=\SSS^2$ where the square roots ensure that $x$ does not describe relative areas but rather relative diameters of fields. This representation is confined to the $\SSS^2$ part in the first octant and every sample shows a distinct accumulation of points in the $x_2 = 0$ plane, corresponding to cells with only one orientation field. As common with biological data, especially from primary cells, their variance is rather high. In consequence, great circle fits are more robust under bootstrapping than small circle fits and we use the nested two-sample tests for PNGS with the following null hypothesis.
\begin{description}
\item[$H_0$:] hMSC orientation and ordering measured by random loci on $\SSS^2$ as above does not change between successive time points.
\end{description}

\begin{table}[!hb]
  \caption{\textit{Displaying p-values of two-sample tests for PNGS of filament orientation field distribution data. The test uses $B=1000$ bootstrap samples, therefore the penultimate p-value is $10^{-3}$.}}
  \begin{center}
    \begin{tabular}{|l|c|c||c|c|}
    \hline
    \rule{0pt}{2.6ex} Time     & \multicolumn{2}{|c||}{\rule{0pt}{2.6ex} nested great circle mean} & \multicolumn{2}{|c|}{\rule{0pt}{2.6ex} jointly great circle and nested mean} \\ \hline
    \rule{0pt}{2.6ex} Gel      & 1~kPa      & 10~kPa and 30~kPa & 1~kPa      & 10~kPa and 30~kPa  \\ \hline
    \rule{0pt}{2.6ex}   4h vs. 8h & $0.120$    & $<10^{-3}$               & $0.308$    & $<10^{-3}$ \\
    \rule{0pt}{2.6ex}  8h vs. 12h & $<10^{-3}$ & $<10^{-3}$               & $0.024$    & $<10^{-3}$ \\
    \rule{0pt}{2.6ex} 12h vs. 16h & $0.126$    & $<10^{-3}$               & $0.008$    & $<10^{-3}$ \\
    \rule{0pt}{2.6ex} 16h vs. 20h & $0.468$    & $0.626$                  & $0.494$    & $0.462$    \\
    \rule{0pt}{2.6ex} 20h vs. 24h & $<10^{-3}$ & $<10^{-3}$               & $<10^{-3}$ & $0.014$    \\ \hline
    \end{tabular}
  \end{center}
  \label{table_fields}
\end{table}

\textbf{Results.} As visible in Table \ref{table_fields}, while for hMSCs on harder gels (10~kPa and 30~kPa), nested means and the joint descriptor of nested mean and great circle change over each 4 hour interval until 16 hours -- for both the null hypothesis is rejected at the highest level possible -- similar changes are less clearly visible for hMSCs on the soft gel (1~kPa) between the intervals between 8 and 16 hours and not at all visible for the first time interval. Strikingly, for hMSCs on all gels, no changes seem to occur between 16 and 20 hours. In contrast, in the final interval between 20 and 24 hours, nested means and great circles clearly change for hMSCs on the soft gel -- rejecting the null hypothesis at the highest level possible. This effect is also there for the nested mean of hMSCs on the harder gels, but not as clearly visible for the joint descriptor including the circle. 

\begin{figure}[!ht]
  \centering
  \subcaptionbox{1~kPa\label{fields_1kpa}}[0.4\textwidth]{\includegraphics[width=0.4\textwidth, clip=true, trim=3cm 0cm 3cm 0cm]{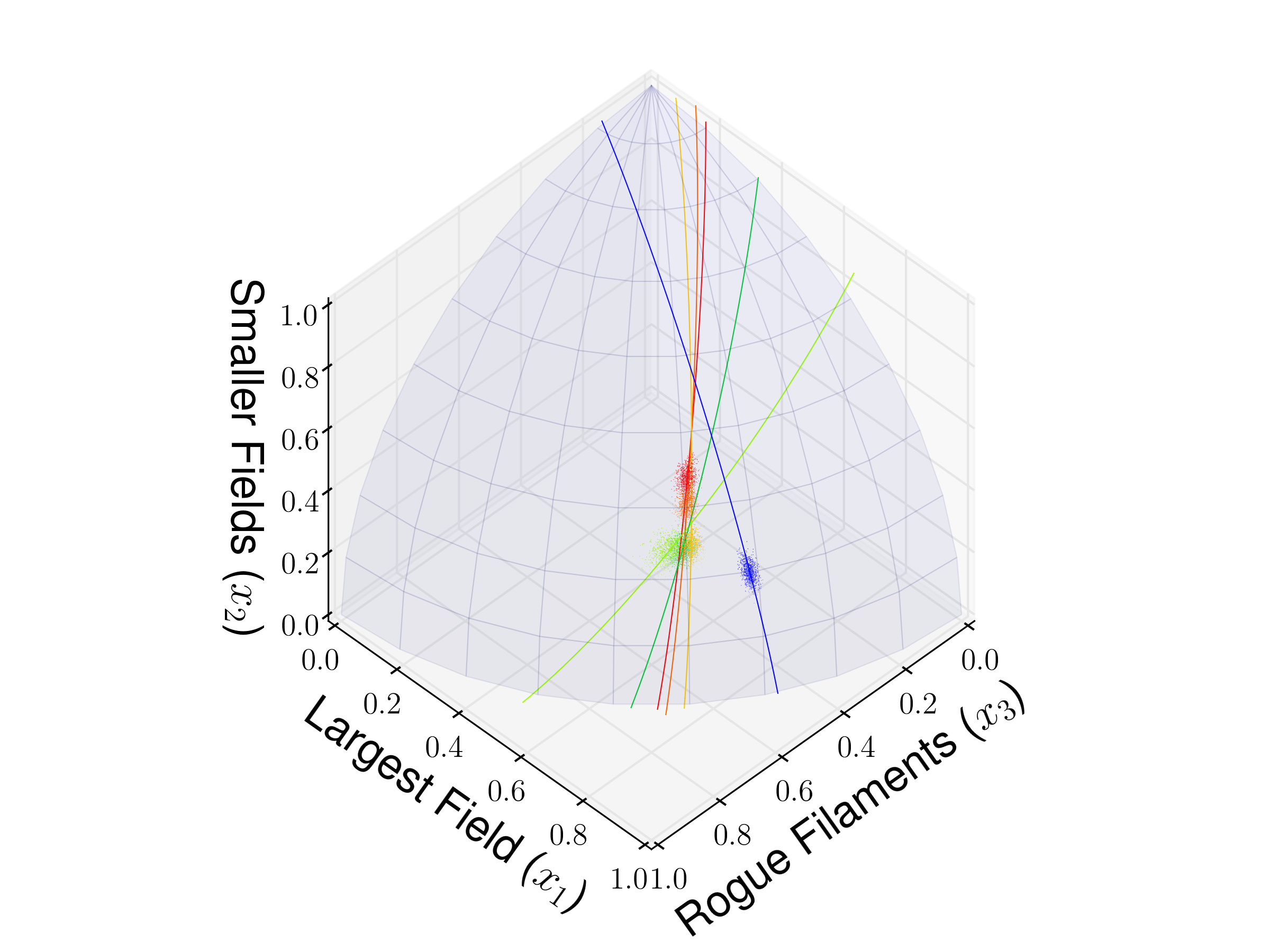}}
  \includegraphics[width=0.18\textwidth, clip=true, trim=4.5cm 0cm 4.5cm 1cm]{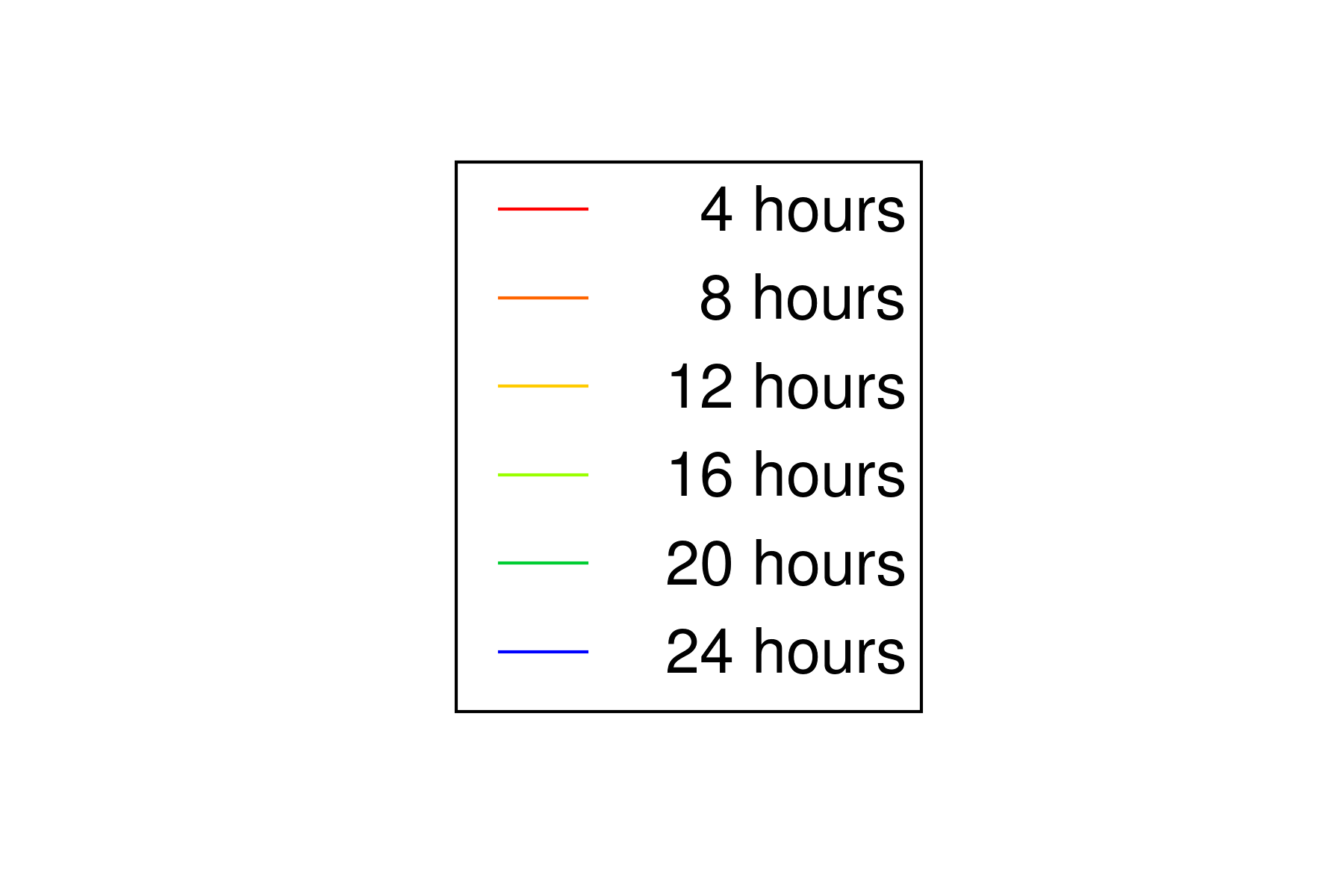}
  \subcaptionbox{10~kPa and 30~kPa\label{fields_pooled}}[0.4\textwidth]{\includegraphics[width=0.4\textwidth, clip=true, trim=3cm 0cm 3cm 0cm]{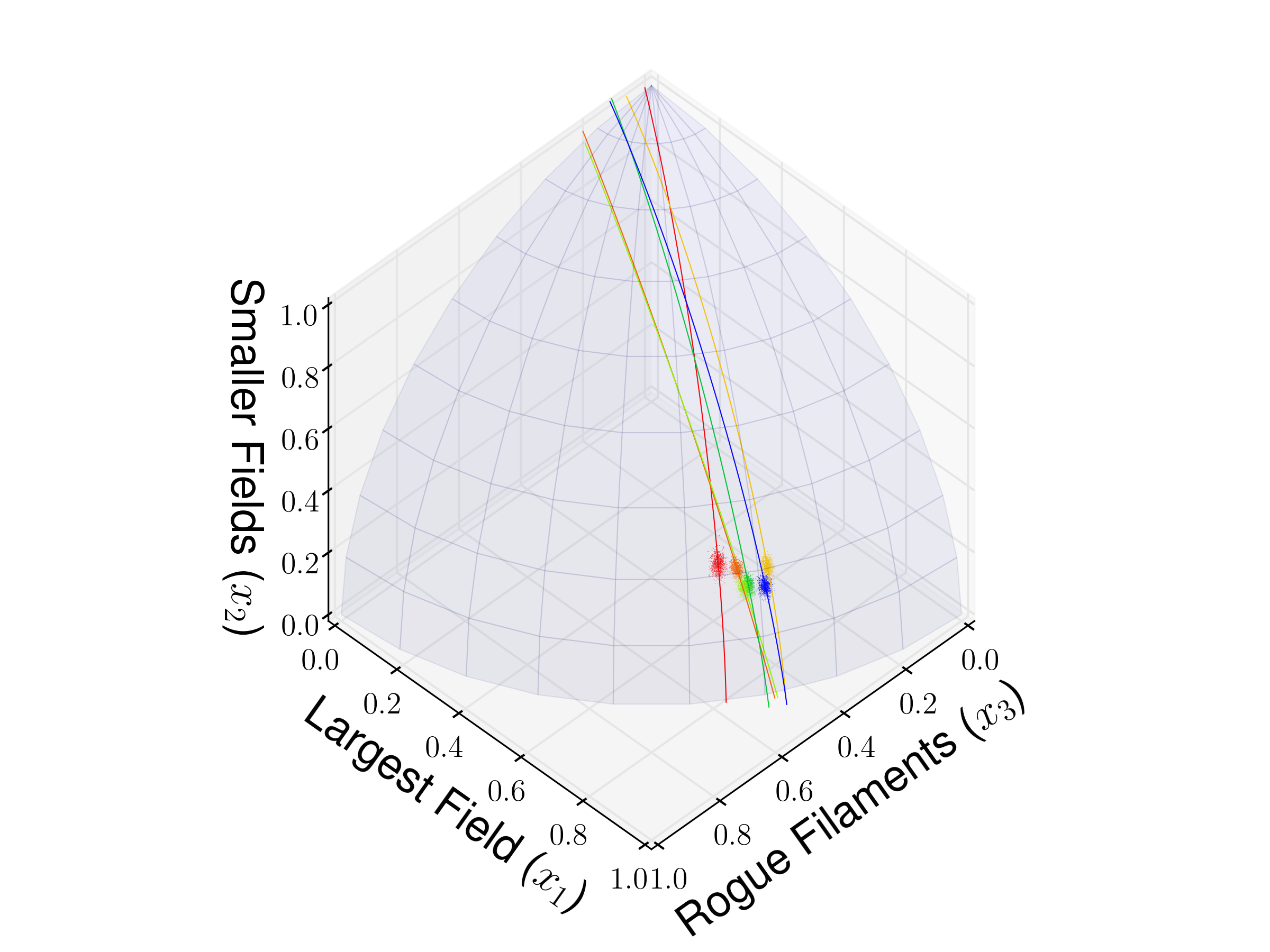}}
  \caption{\it Spherical representations of bootstrapped BNFDs (nested means on mean great circles) for the four data set at 6 time points. 
  \label{fig_fields}}
\end{figure}

\begin{table}[!ht]
  \caption{\textit{Displaying p-values of two-sample tests for PNGS of filament orientation field distribution data for all time points. We use $B=1000$ bootstrap samples, thus the penultimate p-value is $10^{-3}$.}}
  \begin{center}
    \begin{tabular}{|l|c|c|}
    \hline  
    Gels &  \multicolumn{2}{|c|}{\rule{0pt}{2.6ex}1~kPa vs. 10~kPa and 30~kPa}\\ \hline
    Time & nested great circle mean &  jointly great circle and nested mean\\ \hline
      4 h& $< 10^{-3}$ & $< 10^{-3}$\\
     10 h& $< 10^{-3}$ & $< 10^{-3}$\\
     16 h& $< 10^{-3}$ & $< 10^{-3}$\\
     20 h& $< 10^{-3}$ & $< 10^{-3}$\\
     24 h& $0.010$ & $0.061 $\\ \hline
    \end{tabular}
\end{center} 
\label{24hrs_table}
\end{table}

Visualization in Figure \ref{fig_fields} reveals further details. As seen from the loci of the nested means, hMSCs on the soft gel (Figure \ref{fields_1kpa}) tend to loose minor orientation field filaments with a nearly constant ratio of large orientation field filaments and rogue filaments until the \emph{critical slot}, the time interval between 16 and 20 hours. Their great circles, indicating the direction of largest spread, change at the beginning of the critical slot, suggesting that the major variation there occurs in the amount of rogue filaments. While, until the critical slot, the temporal motion of nested means for 1~kPa is mainly vertical, the corresponding motion for the hMSCs on harder gels (cf. Figure \ref{fields_pooled}) is horizontal, indicating that the number of rogue filaments decreases in favor of the main orientation field. Curiously for the nested means, there is also a sharp drop in height at the beginning of the critical slot as well as a backward horizontal motion. After the critical slot, hMSCs seem to continue the direction of their previous journey, at a lower smaller fields' level, though. In contrast, for the hMSCs on the soft gel, the critical slot seems to represent a true change point since afterward, the nested mean travels not much longer towards reducing the smaller fields, but like hMSCs on harder gels, mainly reduces the number of rogue filaments. Indeed, taking into account the auxiliary mesh lines, it can be seen that descriptors are rather close at time 24 hours, cf. Table \ref{24hrs_table}, where, in contrast they are rather far away from each other for all other time points.

\textbf{Discussion.} We conclude that hMSCs react clearly distinctly and differently on both gels already for short time intervals, where at the critical time slot some kind of reboot happens. A generic candidate for this effect is cell division. As all cells used in the experiments were thawed at the same time (72 hours before seeding) and treated identically, cell division is expected to occur at similar (at least for each environment) time points. Dividing cells completely reorganize their cell skeleton which would explain the change point found. In particular, it seems that due to cell division, the time point 24 hours (as used in \cite{ZemelRehfeldBrownDischerSafran2010NatPhys}) may not be ideal if differences in hMSCs differentiation due to different Young's moduli are to be detected. Our results clearly warrant further analysis using higher time resolution, in particular time resolved in-vivo imaging, that among others, allow to register cell division times.

\section*{Acknowledgment}

We thank Rabi Bhattacharya and Vic Patrangenaru for their valuable comments on the bootstrap and our collaborators Florian Rehfeldt and Carina Wollnik for their stem cell data. The authors also gratefully acknowledge DFG HU 1575/4, DFG CRC 755 and the Niedersachsen Vorab of the Volkswagen Foundation.

\section{Appendix}

\subsection{Completing the Proof of Theorem \ref{SC:thm}}

We continue to use the notation introduced in the sketch of the proof right after Theorem \ref{SC:thm}, in particular, recall that $p$ is the last element in $f$. First we show a crucial Lemma. 
\begin{Lem}\label{basic-Ziezold:lem}
  Fix $s\in S_p$. Then there is a measurable set $\Omega_s \subset \Omega$ with $\Prb(\Omega_s) = 1$ such that
  \begin{align*}
    F_{n,f_n}(s^{p_n}) \to F_f(s)\mbox{ and } \frac{1}{n}\sum_{i=1}^n \rho_{p_n}(\pi_{f_n}\circ X_i,s^{p_n}) \to \mathbb E[\rho_{p}(\pi_{f}\circ X,s)] 
  \end{align*}
\end{Lem}

\begin{proof}
  \begin{align*}
    F_{n,f_n}(s^{p_n}) &= \frac{1}{n}\sum_{i=1}^n \rho_{{p_n}}(\pi_{f_n}\circ X_i,s^{p_n})^2\\
    &= \frac{1}{n}\sum_{i=1}^n \rho_{{p}}(\pi_{f}\circ X_i,s)^2 + \frac{2}{n}\sum_{i=1}^n \rho_{{p}}(\pi_{f}\circ X_i,s) h_i(f,f_n,s,s^{p_n}) + \frac{1}{n}\sum_{i=1}^n h_i(f,f_n,s,s^{p_n})^2 
  \end{align*}
  with
  \begin{align*}
    h_i(f,f_n,s,s^{p_n}) = \rho_{{p_n}}(\pi_{f_n}\circ X_i,s^{p_n}) - \rho_{{p}}(\pi_{f}\circ X_i,s)
  \end{align*}
  which, in conjunction with Assumption \ref{projection-close:as} and the induction hypothesis $d(f,f_n) \to 0$ a.s. can be made arbitrary small a.s. due to Assumption \ref{uniform-link:as}. In consequence of the usual strong law the first assertion follows. The second follows with the same argument.
\end{proof}

\textbf{Showing (\ref{Ziezold-sc:eq}).}

Having established Lemma \ref{basic-Ziezold:lem}, in principle we can now follow the steps laid out by \cite{Z77}. They are, however, more intrigued in our endeavor. By hypothesis we have $d(f_n,f) \to 0$ a.s.

By separability of $P_j$ it follows at once from Lemma \ref{basic-Ziezold:lem} that there is a measurable set $\Omega'\subset \Omega$ with $\mathbb P(\Omega')=1$ and a dense subset $\{s_k:k\in \mathbb N\} \subset S_p$ such that
\begin{align}\label{dense-as-conv1:eq}
  \left.\begin{array}{l}F_{n,f_n}(s^{p_n}) \to F_f(s)\mbox{ and }\\ \frac{1}{n}\sum_{i=1}^n \rho_{{p_n}}(\pi_{f_n}\circ X_i,s^{p_n}) \to \mathbb E[\rho_{p}(\pi_{f}\circ X_i,s)]\end{array}\right\} \quad \forall s \in \{s_k:k\in \mathbb N\}\mbox{ and }\omega\in \Omega'\,.
\end{align}

In order to obtain (\ref{dense-as-conv1:eq}) for all $s\in S_p$, consider $s_n,s'_n \in S_{p_n}$ and the following estimates.
\begin{align}\label{dense-as-conv2:ineq}
  \MoveEqLeft \left|F_{n,f_n}(s'_n) - F_{n,f_n}(s_n)\right|\nonumber\\ 
  &\leq \frac{1}{n}\sum_{i=1}^n \left(\rho_{p_n}(\pi_{f_n}\circ X_i,s'_n) + \rho_{p_n}(\pi_{f_n}\circ X_i,s_n)\right)
  \left|\rho_{p_n}(\pi_{f_n}\circ X_i,s'_n) - \rho_{p_n}(\pi_{f_n}\circ X_i,s_n)\right| \nonumber\\
  &=\left\{
  \begin{array}{l}
    \frac{1}{n}\sum_{i=1}^n \left(2\rho_{p_n}(\pi_{f_n}\circ X_i,s'_n) + g_{p_n}(\pi_{f_n}\circ X_i,s_n,s'_n)\right)
    \big|g_{p_n}(\pi_{f_n}\circ X_i,s_n,s'_n)\big|\\ 
    \frac{1}{n}\sum_{i=1}^n \left(2\rho_{p_n}(\pi_{f_n}\circ X_i,s_n) + g_{p_n}(\pi_{f_n}\circ X_i,s'_n,s_n)\right)\big|g_{p_n}(\pi_{f_n}\circ X_i,s_n,s'_n)\big|
  \end{array}
  \right.
\end{align}
with
\begin{align*}
  g_{p_n}(\pi_{f_n}\circ X_i,{s'}^{p_n},s^{p_n}) = \rho_{p_n}(\pi_{f_n}\circ X_i,{s'}^{p_n}) - \rho_{p_n}(\pi_{f_n}\circ X_i,{s}^{p_n})\,.
\end{align*}
Now, w.l.o.g., consider $s_k \to s$ (which implies that $d_j(s_k,s) \to 0$) for which (\ref{dense-as-conv1:eq}) is valid. Using twice the first line in (\ref{dense-as-conv2:ineq}) for $s_n = s^{p_n}$ and $s'_n=s^{p_n}_k$ we obtain
\begin{gather*}
  F_{n,f_n}({s_k}^{p_n}) - \frac{1}{n}\sum_{i=1}^n \left(2\rho_{p_n}(\pi_{f_n}\circ X_i,{s_k}^{p_n}) + g_{p_n}(\pi_{f_n}\circ X_i,s^{p_n},{s_k}^{p_n})\right)
  \big|g_{p_n}(\pi_{f_n}\circ X_i,s^{p_n},{s_k}^{p_n})\big|\\
  \leq F_{n,f_n}({s}^{p_n}) \leq \\
  F_{n,f_n}({s_k}^{p_n}) + \frac{1}{n}\sum_{i=1}^n \left(2\rho_{p_n}(\pi_{f_n}\circ X_i,{s_k}^{p_n}) + g_{p_n}(\pi_{f_n}\circ X_i,s^{p_n},{s_k}^{p_n})\right)
  \big|g_{p_n}(\pi_{f_n}\circ X_i,s^{p_n},{s_k}^{p_n})\big|
\end{gather*}
Due to Assumption \ref{uniform-link:as} and the strong law (\ref{dense-as-conv1:eq}) (and the argument applied in the proof of Lemma \ref{basic-Ziezold:lem}), for every $\epsilon>0$ there is $K=K(\epsilon)\in \mathbb N$ such that for all $k\geq K$ we have
\begin{align*}
  F_f(s_k) - 2\big(\mathbb E[\rho_p(\pi_f\circ X,s_k)] + \epsilon\big)\epsilon & \leq\lim\inf_{n\to \infty} F_{n,f_n}({s}^{p_n})\\
  &\leq \lim\sup_{n\to \infty} F_{n,f_n}({s}^{p_n})\\
  &\leq F_f(s_k) + 2\big(\mathbb E[\rho_p(\pi_f\circ X,s_k)] + \epsilon\big)\epsilon
\end{align*}
for all $\omega \in \Omega'$. Taking into account the continuity of $F_f$, letting $\epsilon \to 0$ yields
\begin{align}\label{as-conv3:eq}
  F_{n,f_n}(s^{p_n}) \to F_f(s) \quad \forall s \in S_p\mbox{ and }\omega\in \Omega'\quad\mbox{ for }n\to \infty\,.
\end{align}
Similarly we see that
\begin{align}\label{as-conv4:eq}
  \frac{1}{n}\sum_{i=1}^n \rho_{p_n}(\pi_{f_n}\circ X_i,s^{p_n}) \to \mathbb E[\rho_p(\pi_{f}\circ X_i,s)]\quad \forall s \in S_p\mbox{ and }\omega\in \Omega'\quad\mbox{ for }n\to \infty\,.
\end{align}

Next, we consider a sequence $S_{p_n}\ni s_n \to s\in S_p$. Note that in consequence of Assumption \ref{almost-triangle:as} we have that
\begin{align*}
  d_j(s_n,s^{p_n}) \to 0\,.
\end{align*}
Using the bottom line of (\ref{dense-as-conv2:ineq}) yields that
\begin{align*}
  \MoveEqLeft\big|F_{n,f_n}(s_n) - F_{n,f_n}(s^{p_n})\big|\\
  &\leq \frac{1}{n}\sum_{i=1}^n \left(2\rho_{p_n}(\pi_{f_n}\circ X_i,{s}^{p_n}) + g_{p_n}(\pi_{f_n}\circ X_i,s_n,{s}^{p_n})\right)
  \big|g_{p_n}(\pi_{f_n}\circ X_i,s_n,{s}^{p_n})\big|
  &\to 0\quad \forall \omega\in \Omega'
\end{align*}
with the same $\Omega'$ for all $s\in S_p$ due to (\ref{as-conv4:eq}). Hence, in consequence of this and (\ref{as-conv3:eq}), for all $S_{p_n}\ni s_n \to s\in S_p$ we have that
\begin{align}\label{as-conv5:eq}
  \big|F_{n,f_n}(s_n) - F_f(s)\big| \leq \big|F_{n,f_n}(s_n) - F_{n,f_n}(s^{p_n})\big| + \big| F_{n,f_n}(s^{p_n})- F_f(s)\big| \to 0\,.
\end{align}
for all $\omega\in \Omega'$.

Finally let us show
\begin{align}\label{Ziez3:eq}
  \mbox{if $\cap_{n=1}^\infty \overline{\cup_{k=n}^\infty E^{f_k}_k}\neq \emptyset$ then $\lw_{n,f_n} \to \lw_f$ a.s.}
\end{align}
Note that Assertion (\ref{Ziezold-sc:eq}) is trivial in case of $\cap_{n=1}^\infty \overline{\cup_{k=n}^\infty E_k^{f_k}}=\emptyset$. Otherwise, for ease of notation let $B_n := \cup_{k=n}^\infty E^{f_k}_k$, $\overline{B_n}\searrow B:=\cap_{n=1}^\infty \overline{B_n}$, $b \in B$. Then $b \in \overline{B_n}$ for all $n\in\mathbb N$. Hence, there is a sequence $b_n \in B_n$, $b_n \to b$. Moreover, there is a sequence $k_n$ such that $b_n =s_{k_n} \in E^{f_{k_n}}_{k_n}$ for a suitable $k_n \geq n$. Then $\lw_{n_k,f_{n_k}}=F_{n_k,f_{n_k}}(s_{n_k})\to F_f(b)\geq \lw$ by (\ref{as-conv5:eq}) a.s.. On the other hand, by Lemma \ref{basic-Ziezold:lem} for arbitrary fixed $s\in S_p$, there is a sequence $\epsilon_{n}\to 0$ such that $F_f(s)\geq F_{n,f_n}(s^{p_n})- \epsilon_{n}\geq \lw_{n,f_n} - \epsilon_{n}$. First letting $n\to \infty$ and then considering the infimum over $s\in S_p$ yields
\begin{align*}
  \lw_f \geq\limsup_{n\to\infty} \lw_{n,f_n}\,.
\end{align*}
In consequence
\begin{align}\label{lw-lim-sup:ineq}
  \limsup_{n\to\infty}\lw_{n,f_n} \geq F_f(b)\geq \lw_f \geq \limsup_{n\to\infty}\lw_{n,f_n} \mbox{ a.s.}
\end{align}
In particular we have shown that $\lw_f =F_f(b)$ which means that $b=s^*$ thus completing the proof of (\ref{Ziezold-sc:eq})

\textbf{Proof of (\ref{BP-SC:eq}).}
Using the notation of the previous proof of (\ref{Ziezold-sc:eq}), let $s_n\in S_{p_n}$ and consider $r_n = d_j(s^*,s_n)$. If the assertion (\ref{BP-SC:eq}) was false, there would be a measurable set $A\subset \Omega$ with $\Prb(A)>0$ such that for every $\omega\in A$ there is $r_0(\omega) > 0$ and $r_n(\omega)\geq r_0(\omega)>0$.

First, we claim that $\Prb(B)=0$ with $B=\{\omega \in A: s_n(\omega)\mbox{ has a cluster point}\}$. 

For if $\omega \in B$ with $s_n(\omega) \to \widetilde{s}(\omega)\in S_p$, then by continuity $d_j(s^*(\omega),\widetilde{s}(\omega))\geq r_0(\omega)>0$ which in conjunction with (\ref{Ziezold-sc:eq}),
\begin{align*}
  \widetilde{s}\in \bigcap_{n=1}^\infty \overline{\bigcup_{k=n}^\infty E^{f_k}_k} \subset \{s^*\}\quad a.s\,.
\end{align*}
implies that $\Prb(B)=0$.

In consequence of the Heine Borel property we have thus $r_n \to \infty $ for all $\omega \in A\setminus B$.
Since $\E[\rho_{p}(\pi_f\circ X,s)^2] <\infty$ for all $s\in S_p$, there is a $C>0$ such that
\begin{align*}
  \Prb\{\omega\in A: \rho_{p}(\pi_f\circ X,s)<C\}>0
\end{align*}
Hence, in consequence of Assumption \ref{coercivity:as} we have thus a subset $A'\subset A\setminus B$ with $\Prb(A') >0$ such that for all $\omega\in A'$, due to the usual strong law,
\begin{align*}
  \lw \geq \lw_{n,f_n} = F_{n,f_n}(s_n) \geq \frac{1}{n}\sum_{i=1}^{n} 1_{\{\omega\in A:\rho_{p}( \pi_p\circ X_{i}(\omega),s)<C\}}~\rho_{{p_n}}( \pi_{f_n}\circ X_{i}(\omega),s_n)^2 ~\to~\infty~a.s.
\end{align*}
This is a contradiction to (\ref{lw-lim-sup:ineq}). This yields (\ref{BP-SC:eq}) completing the proof of Theorem \ref{SC:thm}.


\end{document}